\documentclass[11pt, fleqno]{amsart}

\usepackage{etex}
\reserveinserts{28}

\usepackage{epsf}
\usepackage{geometry}
\usepackage{listings} 
\usepackage{multirow}
\usepackage{enumerate}

\usepackage[pdftex]{graphicx}
\usepackage{amscd}
\usepackage[pdftex]{color} 
\usepackage[pdftex,colorlinks]{hyperref}
\usepackage{graphicx,pst-eps,epstopdf}
\usepackage{lscape}
\usepackage{indentfirst}
\usepackage{latexsym}
\usepackage{amsmath, amsfonts, amssymb,mathrsfs}
\usepackage{subfigure,pstricks,pst-node}
\usepackage{pst-eps,epstopdf}
\usepackage{verbatim}
\usepackage{tikz}
\usepackage{caption}
\usepackage{nicefrac}
\usepackage{float}
\usepackage{bm}
\usepackage{fourier}

\usepackage[titletoc,toc,title]{appendix}

\hypersetup{
    bookmarks=true,         
    unicode=false,          
    pdftoolbar=true,        
    pdfmenubar=true,        
    pdffitwindow=true,      
    pdftitle={},    
    pdfauthor={}
    pdfsubject={},   
    pdfnewwindow=true,      
    pdfkeywords={}, 
    colorlinks=true,       
    linkcolor=red,          
    citecolor=blue,        
    filecolor=magenta,      
    urlcolor=cyan           
}

\newtheorem{theorem}{Theorem}[section]

\newtheorem{remark}{Remark}[section]
\newtheorem{lemma}{Lemma}[section]

\newtheorem{algorithm}{Algorithm}[section]

\numberwithin{equation}{section} \numberwithin{table}{section}
\numberwithin{figure}{section}

\DeclareMathOperator*{\esssup}{ess\,sup}

\everymath{\displaystyle}

\title[Error estimates]{Error estimates for structure-preserving discretization of the incompressible MHD system
\footnote{This material
    is based upon work supported in part by the US Department of
    Energy Office of Science, Office of Advanced Scientific Computing
    Research, Applied Mathematics program under Award Number
    DE-SC0014400 and by Beijing International Center for Mathematical
    Research of Peking University, China.}
}

\author{Yicong Ma}
\address{Department of Mathematics,The Pennsylvania State University, University Park, PA 16802, USA}
\email{yxm147@psu.edu}
\thanks{}

\author{Jinchao Xu}
\address{Department of Mathematics,The Pennsylvania State University, University Park, PA 16802, USA}
\email{jinchao@psu.edu}
\thanks{}

\author{Guodong Zhang}
\address{School of Mathematics and Statistics, Xi'an Jiaotong University, Xi'an 710049, China}
\email{gdzhang2014@gmail.com}
\thanks{}

\begin{document}
\maketitle

\begin{abstract}
In this paper, we carry out the error analysis for the structure-preserving discretization of the incompressible MHD system. This system, as a coupled system of Navier-Stokes equations and Maxwell's equations, is nonlinear. We use its energy estimate and the underlying physical structure to facilitate the error analysis. Under certain CFL conditions, we prove the optimal order of convergence. To support the theoretical results, we also present numerical tests.
\end{abstract}



\section{Introduction}
An incompressible magnetohydrodynamic (MHD) system is a coupled partial differential equation system
resulting from the incompressible Navier-Stokes equations and the (reduced)
Maxwell's equations. Assuming $\Omega \subset \mathbb{R}^{3}$ is a simply connected open-bounded domain with Lipschitz boundary, the model problem we consider is
\begin{align}
\begin{cases}
& \bm{u}_{t}
+ ( \bm{u} \cdot \nabla) \bm{u}
- R_{e}^{-1} \Delta \bm{u}
- s \bm{j} \times \bm{B}
+ \nabla p
= \bm{f} , \\
& \bm{B}_{t}
+ \nabla \times \bm{E} = \bm{0} , \\
& \bm{j} - R_{m}^{-1} \nabla \times \mu_{r}^{-1} \bm{B}
 = \bm{0} , \\
& \sigma_{r} ( \bm{E} + \bm{u} \times \bm{B} )
= \bm{j} , \\
& \nabla \cdot \bm{u} = 0.
\end{cases}\label{eq:dimensionless}
\end{align}
The coefficients in this system are the Reynolds number $R_{e}$, the magnetic
Reynolds number $R_{m}$, the coupling number $s$, the relative electric
conductivity $\sigma_{r}$, and the relative magnetic permeability
$\mu_{r}$. The initial conditions for this set of equations are
\begin{align*}
& \bm{u}(\bm{x}, 0) =
  \bm{u}_{0}(\bm{x}),
  \quad
\bm{B}(\bm{x}, 0) =
  \bm{B}_{0}(\bm{x}),
  \quad \forall \bm{x} \in \Omega,
\end{align*}
and the boundary conditions are
\begin{align*}
  & \bm{u} = \bm{0},
  \quad
  \bm{n} \times \bm{E} = \bm{0},
  \quad
  \bm{n} \cdot \bm{B} = 0,
  \quad \forall \bm{x} \in \partial \Omega, \quad t > 0.
\end{align*}

As discussed in the literature, the variables $\bm{u}$, $\bm{B}$ and $p$, once known, uniquely determine $\bm{E}$ and $\bm{j}$. There are many different numerical methods to discretize MHD. We now briefly examine some existing literature on some of the numerical methods and their error analysis for two types of MHD systems: the stationary MHD system \cite{Gunzburger.M;Meir.A;Peterson.J.1991a,Schotzau.D.2004a} and the evolutionary MHD system \cite{Hu.K;Ma.Y;Xu.J.2014a,Prohl.A.2008a,He.Y.2015a}.

For the stationary MHD system, Gunzburger, Meir and Peterson \cite{Gunzburger.M;Meir.A;Peterson.J.1991a} propose a formulation with $H^{1}$ finite element discretization for the magnetic field, and analyze its well-posedness and convergent behavior. Sch\"{o}tzau \cite{Schotzau.D.2004a}, who also works on the stationary MHD system, proposes a new formulation with $H(\mathrm{curl})$ discretization for the magnetic field, and proves its well-posedness and the optimal order of convergence. There are also many other methods for stationary problems, for example, \cite{Gerbeau.J.2000a,Guermond.J;Minev.P.2003a,Shadid.J;Pawlowski.R;Banks.J;Chacon.L;Lin.P;Tuminaro.R.2010a}.

For the evolutionary MHD system, Prohl \cite{Prohl.A.2008a} studies the coupled and decoupled schemes based on  $H(\mathrm{curl})$ conforming discretization of the magnetic field. He proves that the discrete solution converges to the weak solution under a strong Courant-Friedrichs-Lewy (CFL) condition; that is, $k \leq C h^{3}$ ($k$ stands for the time step size,  and $h$ for the mesh size). And He \cite{He.Y.2015a} studies the MHD system on a regular domain with $H^{1}$ conforming discretization of the magnetic field. He proves an unconditional optimal order of convergence.

In this paper, we study the convergence property of a structure-preserving discretization presented in \cite{Hu.K;Ma.Y;Xu.J.2014a}. This method is based on the mixed formulation \cite{Boffi.D;Brezzi.F;Fortin.M.2013a}, which comes from the idea of FEEC (finite element exterior calculus) \cite{Arnold.D;Falk.R;Winther.R.2006a,Arnold.D;Falk.R;Winther.R.2010a} and DEC (discrete exterior calculus) \cite{Bossavit.A.2005a}. $H(\mathrm{curl})$ and $H(\mathrm{div})$ conforming finite element discretization are used for the electric field and the magnetic field respectively. The advantage of this approach is that the important Gauss's law for magnetic field is preserved exactly on the discrete level. Moreover, the incompressible MHD system we focus on is a time-dependent nonlinear problem. Therefore, to conduct the error analysis, we work on an evolutionary nonlinear saddle point problem. Before approaching the detailed analysis, we briefly review the existing literatures for the error estimates of the (evolutionary) saddle point problems and nonlinear problems.

Abstract error estimates exist for standard (linear, non-evolutionary) saddle point systems \cite{Boffi.D;Brezzi.F;Fortin.M.2013a}. Optimal order of convergence is ensured by the well-posedness of the discretization system and the approximation property of the finite element space.
For the evolutionary saddle point problem, Boffi and Gastaldi \cite{Boffi.D;Gastaldi.L.2004a} build a general framework for the semi-discretization of the evolutionary (linear) saddle point problem and provide sufficient conditions for a good approximation in the natural functional spaces.

For nonlinear saddle point problems, no abstract error estimate framework can be found in the literature. But various techniques have been developed for specific problems. For example, Temann \cite{Teman.R.1977a} discusses the theory and numerical methods for NS equations. Heywood and Rannacher \cite{Heywood.J;Rannacher.R.1982a,Heywood.J;Rannacher.R.1986a,Heywood.J;Rannacher.R.1988a,Heywood.J;Rannacher.R.1990a} discuss the stability and error estimates of both semi-discretization and full discretization schemes for the NS systems. He \cite{He.Y.2008a} study linearized implicit-explicit schemes for this model.

General error estimates exist for nonlinear parabolic and elliptic problems. Thom\'{e}e et al. \cite{Schatz.A;Thomee.V;Wahlbin.L.1980a,Johnson.C;Larsson.S;Thomee.V;Wahlbin.L.1987a,Huang.M;Thomee.V.1981a,Huang.M;Thomee.V.1982a,Crouzeix.M;Thomee.V;Wahlbin.L.1989a,Thomee.V.1997a} investigate the error estimates of nonlinear parabolic problems intensively. Xu \cite{Xu.J.1996b} uses the priori $W^{1,\infty}$ estimate to derive the $W^{m,p}$ error estimates of a general nonlinear elliptic problem. Brezzi, Rappaz and Raviart \cite{Brezzi.F;Rappaz.J;Raviart.P.1980a,Brezzi.F;Rappaz.J;Raviart.P.1981a,Brezzi.F;Rappaz.J;Raviart.P.1982a} build an abstract theory for finite element approximation of nonlinear problems.



Due to the nonlinearity and the loss of coercivity of the MHD model, the error estimate
becomes difficult. 
To estimate the error of nonlinear problems, we usually need to prove that the $L^{\infty}$ norm (or a stronger norm) of the numerical solution is bounded. Generally, there are two ways to obtain this bound, one is using the mathematical induction method \cite{He.Y.2008a,He.Y;Sun.W.2007b}, the other is introducing a semi-discrete problem \cite{B.Li;W.Sun2013a,B.Li;W.Sun2013b}. Moreover, due to the loss of coercivity, we cannot use Cea's Lemma to derive the error estimates directly.  

In our analysis, we take advantage of the energy estimate of the structure-preserving discretization instead of estimating the $L^{\infty}$ norm of the numerical solution. We prove the unconditional error estimates for
the velocity $\bm{u}$, the magnetic field $\bm{B}$, and the volume
current density $\bm{j}$. And under certain constrains on 
the time-step size, we derive error estimates for the electric
field $\bm{E}$ and the pressure $p$.  Numerical tests support the theoretical results.

We organize this paper as follows. In \S \ref{sec:notation}, we
introduce useful notation for our analysis. In ~\S
\ref{sec:discretization}, we go over the discretization schemes and
their energy estimates. We carry out detailed error estimates in \S
\ref{sec:error_estimates}, and present numerical experiments to
demonstrate the optimal order of convergence in \S
\ref{sec:numer_tests}.

\section{Magnetohydrodynamics model}\label{sec:notation}
In this section, we introduce some notation which follows mostly
\cite{Hu.K;Ma.Y;Xu.J.2014a}. We first define the usual $L^{2}$ inner
product
\begin{align*}
& (u, v) = \int_{\Omega} u \cdot v dx,
\end{align*}
and the $L^{2}$ norm
\begin{align*}
& \Vert u \Vert = \left(\int_{\Omega} \lvert u \rvert^{2} dx \right)^{1/2}.
\end{align*}
For the sake of simplicity, we write both $L^{2}(\Omega)$ and $\left[ L^{2}(\Omega) \right]^{3}$ as $L^{2}(\Omega)$.

Given a linear operator $D$, we define
\begin{align*}
& H(D, \Omega) = \left\{ v \in L^{2}(\Omega),~ Dv \in L^{2}(\Omega) \right\},
\end{align*}
and
\begin{align*}
& H_{0}(D, \Omega) = \left\{ v \in H(D, \Omega), ~ t_{D} v = 0 ~ \mbox{on} ~ \partial \Omega \right\},
\end{align*}
where $t_{D}$ is the trace operator defined by
\begin{align*}
  & t_{D} v = \begin{cases}
    v, ~ D = \mathrm{grad}, \\
    n \times v, ~ D = \mathrm{curl}, \\
    n \cdot v, ~ D = \mathrm{div}.
    \end{cases}
\end{align*}
We define
\begin{align*}
& L_{0}^{2}(\Omega) := \left\{v \in L^{2}(\Omega), ~ \int_{\Omega} v = 0  \right\}.
\end{align*}
When $D = \mathrm{grad}$, we typically write $H^{1}(\Omega)$ instead of $H(\mathrm{grad}, \Omega)$, and $H_{0}^{1}(\Omega)$ instead of $H_{0}(\mathrm{grad}, \Omega)$. In the analysis, we also use the spaces $W^{m,p}$ and $H^{-1}$ with norms
\begin{align*}
  & \Vert v \Vert_{m,p} = \left(
  \sum \limits_{ 0 \leq \lvert \alpha \rvert \leq m }
  \int_{\Omega} \lvert D^{\alpha} v \rvert^{p} \right)^{1/p},
  \quad
  \Vert v \Vert_{0,\infty} = \esssup\limits_{x \in \Omega} \lvert v \rvert,
  \quad
  \Vert v \Vert_{-1} = \sup\limits_{\phi \in H_{0}^{1}(\Omega)} \frac{( v, \phi )}{\Vert \nabla \phi \Vert}.
\end{align*}
Another useful norm in the analysis is the discrete $L^{2}( [0,T], \ast )$ norm
\begin{align*}
& \VERT u \VERT_{m, \ast}^{2} = k \sum \limits_{n = 1}^{m} \Vert u^{n} \Vert_{\ast}^{2},
\end{align*}
where $k$ is the time step size (as is the $k$ in the following context). For example, $\VERT \cdot \VERT_{m,0}$ stands for the discrete $L^{2}([0, T], L^{2})$ norm and $\VERT \cdot \VERT_{m,-1}$ for the discrete $L^{2}( [0, T], H^{-1} )$ norm.

Next, we introduce some useful function spaces in the discretization.
\begin{align*}
  & \bm{X} = H_{0}^{1}(\Omega)^{3} \times H_{0}(\mathrm{div}, \Omega) \times H_{0}(\mathrm{curl}, \Omega),
  \quad
  Q = L_{0}^{2}(\Omega), \\
  & \bm{V} = H_{0}^{1}(\Omega)^{3},
  ~ \bm{V}^{d} = H_{0}(\mathrm{div}, \Omega),
  ~ \bm{V}^{c} = H_{0}(\mathrm{curl}, \Omega).
\end{align*}
We use $\bm{W}^{\ast}$ ($\bm{W} = \bm{V}$, $\bm{V}^{d}$, $\bm{V}^{c}$ or $Q$) to denote the dual space of $\bm{W}$, and $\bm{W}_{h}$ to denote the finite element space of $\bm{W}$. The divergence-free subspace of $\bm{V}$ is defined as
\begin{align*}
& \bm{V}^{0} = \left\{
\bm{v} \in \bm{V}, ~ \nabla \cdot \bm{v} = 0
\right\}.
\end{align*}

For the sake of simplicity, we assume that $\mu_{r} = \sigma_{r} = 1$ in the analysis. The Hilbert spaces $\bm{X}$ and $Q$ are equipped with norms $\Vert \cdot \Vert_{ \bm{X} }$ and $\Vert \cdot \Vert_{Q}$, which are defined as
\begin{align*}
& \Vert \bm{\xi} \Vert_{ \bm{X} }^{2}
= \Vert \bm{v} \Vert^{2}_{1}
+ \Vert \bm{C} \Vert^{2}_{ \mathrm{div} }
+ \Vert \bm{F} \Vert^{2}_{ \mathrm{curl} },
\quad
\forall \bm{\xi} = ( \bm{v}, \bm{C}, \bm{F} ) \in \bm{X}, \\
& \Vert q \Vert_{Q}^{2}
= \Vert q \Vert^{2},
\quad \forall q \in Q.
\end{align*}
Here,
\begin{align*}
& \Vert \bm{v} \Vert_{1}^{2} = \Vert \bm{v} \Vert^{2} + \Vert \nabla \bm{v} \Vert^{2},
~ \forall \bm{v} \in H^{1}(\Omega), \\
& \Vert \bm{C} \Vert_{\mathrm{div}}^{2}
= \Vert \bm{C} \Vert^{2}
+ \Vert \nabla \cdot \bm{C} \Vert^{2},
~ \forall \bm{C} \in H(\mathrm{div}, \Omega ),
\\
& \Vert \bm{F} \Vert_{\mathrm{curl}}^{2}
= \Vert \bm{F} \Vert^{2}
+ \Vert \nabla \times \bm{F} \Vert^{2},
~ \forall \bm{F} \in H(\mathrm{curl}, \Omega ) .
\end{align*}
We also use Sobolev space $H^{r}(\mathrm{div}, \Omega)$ and $H^{r}(\mathrm{curl}, \Omega)$, which are defined as
\begin{align*}
& H^{r}(\mathrm{div}, \Omega)
= \left\{
\bm{v} \in H^{r}(\Omega),
~ \nabla \cdot \bm{v} \in H^{r}(\Omega)
\right\},
\\
& H^{r}(\mathrm{curl}, \Omega)
= \left\{
\bm{v} \in H^{r}(\Omega),
~ \nabla \times \bm{v} \in H^{r}(\Omega)
\right\}.
\end{align*}
The corresponding norms are denoted by $\Vert \cdot \Vert_{r, \mathrm{div}}$ and $\Vert \cdot \Vert_{r, \mathrm{curl}}$, which are defined as
\begin{align*}
& \Vert \bm{C} \Vert_{r, \mathrm{div}}^2 = \Vert \bm{C}\Vert_{r,2}^2 +\Vert \nabla \cdot \bm{C}\Vert_{r, 2}^2,
\\
& \Vert \bm{F} \Vert_{r, \mathrm{curl}}^2 =\Vert \bm{F}\Vert_{r,2}^2 +\Vert \nabla \times \bm{F}\Vert_{r, 2}^2.
\end{align*}

To facilitate the analysis, we also introduce a tri-linear form of $\bm{V}$, namely,
\begin{align}
& c(\bm{\phi}, \bm{u}, \bm{v} )
= \dfrac{1}{2}
( \bm{\phi} \cdot \nabla \bm{u}, \bm{v} )
- \dfrac{1}{2} (\bm{\phi} \cdot \nabla \bm{v}, \bm{u} ).
\label{eq:convection_trilinear}
\end{align}

Based on the above notations, the variational formulation for system \eqref{eq:dimensionless} is: find $(\bm{u}, \bm{B}, \bm{E}) \in \bm{X}$ and $p \in Q$ such that for any $(\bm{v}, \bm{C}, \bm{F}) \in \bm{X}$ and $q \in Q$,
\begin{align}
  \begin{cases}
    &  \left( \bm{u}_{t},\bm{v} \right)
    + c( \bm{u}, \bm{u}, \bm{v} )
    + R_{e}^{-1} (\nabla \bm{u}, \nabla \bm{v})
    - s ( \bm{j} \times \bm{B},\bm{v} )
    - (p,\nabla\cdot \bm{v})
    = (\bm{f},\bm{v}), \\
    & \alpha \left( \bm{B}_{t}, \bm{C} \right)
    + \alpha ( \nabla\times \bm{E}, \bm{C}) = 0,   \\
    & s ( \bm{j},\bm{F})
    - \alpha ( \bm{B}, \nabla\times \bm{F})
    = 0,  \\
    & (\nabla\cdot \bm{u}, q)  =  0,
  \end{cases}\label{eq:variational_form}
\end{align}
where $\bm{j} = \bm{E} + \bm{u} \times \bm{B}$, and $\alpha = s/R_{m}$.

\section{Finite element discretization}\label{sec:discretization}
In this section, we briefly go over the finite element discretization of \eqref{eq:dimensionless}. For the temporal discretization, we use the backward Euler method. For the spacial discretization, we recall the formulation of both nonlinear and linearized discretization here. These discretization formulations are reasonable in the sense that they inherit the energy estimate from the continuous level. At the end of this section, we go over their energy estimates.

\begin{algorithm} \label{prob:picard_linearization}
  Find $ (\bm{u}^{n}_{h},
  \bm{B}^{n}_{h}, \bm{E}^{n}_{h}) \in \bm{X}_{h}$ and $ p^{n}_{h} \in Q_{h}$ such that for any $(\bm{v}_{h}, \bm{C}_{h}, \bm{F}_{h} ) \in \bm{X}_{h}$ and $q_{h} \in Q_{h}$,
  \begin{align}
    \begin{cases}
      & ( \bar{\partial} \bm{u}^{n}_{h},\bm{v}_{h} )
      + c( \bm{u}^{n-1}_{h}, \bm{u}^{n}_{h}, \bm{v}_{h} )
     + R_{e}^{-1} (\nabla \bm{u}^{n}_{h}, \nabla \bm{v}_{h})
     - s (\bm{j}^{n}_{h} \times \bm{B}^{n-1}_{h},\bm{v}_{h} ) \\
     & \qquad
     - (p^{n}_{h},\nabla\cdot \bm{v}_{h}) = (\bm{f}^{n},\bm{v}),\\
     & \alpha ( \bar{\partial} \bm{B}^{n}_{h}, \bm{C}_{h} )
     + \alpha (\nabla\times\bm{E}^{n}_{h}, \bm{C}_{h} ) = 0, \\
     & s(\bm{j}^{n}_{h},\bm{F}_{h}) - \alpha ( \bm{B}^{n}_{h},
     \nabla\times \bm{F}_{h}) =0, \\
     & (\nabla\cdot \bm{u}^{n}_{h}, q_{h}) = 0,
    \end{cases}\label{eq:picard_linearization}
  \end{align}
where $\bm{j}^{n}_{h} = \bm{E}^{n}_{h} + \bm{u}^{n}_{h} \times
\bm{B}^{n-1}_{h}$, $ \bar{\partial} \bm{u}_{h}^{n} = k^{-1} ( \bm{u}_{h}^{n} - \bm{u}_{h}^{n-1} ) $, and $ \bar{\partial} \bm{B}_{h}^{n} = k^{-1} ( \bm{B}_{h}^{n} - \bm{B}_{h}^{n-1} ) $.

\end{algorithm}

The above formulation uses linearization as a discretization scheme. In fact, we can discretize the nonlinear system directly and solve the nonlinear equation by Picard or Newton iteration.

\begin{algorithm} \label{prob:picard}
  Find $ (\bm{u}^{n}_{h},
  \bm{B}^{n}_{h}, \bm{E}^{n}_{h}) \in \bm{X}_{h}$ and $ p^{n}_{h} \in Q_{h}$ such that for any $(\bm{v}_{h}, \bm{C}_{h}, \bm{F}_{h} ) \in \bm{X}_{h}$ and $q_{h} \in Q_{h}$,
  \begin{align}
    \begin{cases}
      & ( \bar{\partial} \bm{u}^{n}_{h},\bm{v}_{h} )
      + c( \bm{u}^{n}_{h}, \bm{u}^{n}_{h}, \bm{v}_{h} )
     + R_{e}^{-1} (\nabla \bm{u}^{n}_{h}, \nabla \bm{v}_{h})
     - s (\bm{j}^{n}_{h} \times \bm{B}^{n}_{h},\bm{v}_{h} ) \\
     & \qquad
     - (p^{n}_{h},\nabla\cdot \bm{v}_{h}) = (\bm{f}^{n},\bm{v}),\\
     & \alpha ( \bar{\partial} \bm{B}^{n}_{h}, \bm{C}_{h} )
     + \alpha (\nabla\times\bm{E}^{n}_{h}, \bm{C}_{h} ) = 0, \\
     & s(\bm{j}^{n}_{h},\bm{F}_{h}) - \alpha ( \bm{B}^{n}_{h},
     \nabla\times \bm{F}_{h}) =0, \\
     & (\nabla\cdot \bm{u}^{n}_{h}, q_{h}) = 0,
    \end{cases}\label{eq:fully_picard}
  \end{align}
where $\bm{j}^{n}_{h} = \bm{E}^{n}_{h} + \bm{u}^{n}_{h} \times
\bm{B}^{n}_{h}$.

\end{algorithm}

As mentioned before, these above formulations admit desirable energy estimates. For the sake of completeness, we cite some of these estimates, which are established in \cite{Hu.K;Ma.Y;Xu.J.2014a}.

\begin{theorem}[Energy estimates]
For any $(\bm{u}_{h}^{n}, \bm{B}_{h}^{n}, \bm{E}_{h}^{n} ) \in \bm{X}_{h}$, and $p_{h}^{n} \in Q_{h}$ that satisfies \eqref{eq:fully_picard}, the following energy estimates hold
\begin{align*}
& \max \limits_{0 \leq n \leq m} \left( \Vert \bm{u}_{h}^{n} \Vert^{2}
+ \alpha \Vert \bm{B}_{h}^{n} \Vert^{2} \right)
+ R_{e}^{-1} \VERT \nabla \bm{u}_{h} \VERT_{m, 0}^{2}
+ 2s \VERT \bm{j}_{h} \VERT^{2}_{m, 0}
\leq
\Vert \bm{u}^{0} \Vert^{2}
+ \alpha \Vert \bm{B}^{0} \Vert^{2}
+ R_{e} \VERT \bm{f} \VERT^{2}_{m,-1},
\end{align*}
where $\bm{j}_{h}^{n} = \bm{E}^{n}_{h} + \bm{u}^{n}_{h} \times
\bm{B}^{n}_{h}$.
\end{theorem}
Here, $\bm{u}^{0}$ and $\bm{B}^{0}$ depend on the initial data. A similar energy estimate holds for~\eqref{eq:picard_linearization}.


\section{Error estimates}\label{sec:error_estimates}

Before starting the detailed analysis, first we recall Gronwall's inequality \cite{Heywood.J;Rannacher.R.1990a}, which is an important tool in our analysis.

\begin{theorem}[Gronwall's inequality \cite{Heywood.J;Rannacher.R.1990a}] \label{lemma:Gronwall_ineqn}
Let $k$, $B$, and $a_{i}$, $b_{i}$, $c_{i}$, $\gamma_{i}$, for integers $i \geq 0$, be non-negative numbers such that
\begin{align*}
& a_{n} + k \sum \limits_{ i = 0 }^{n} b_{i}
\leq
k \sum \limits_{i=0}^{n} \gamma_{i} a_{i}
+ k \sum \limits_{i=0}^{n} c_{i} + B,
\quad \forall ~ n \geq 0.
\end{align*}
Suppose that $k \gamma_{i} < 1$ (for all $i$), and set $\sigma_{i} = ( 1 - k \gamma_{i} )^{-1}$, then
\begin{align*}
& a_{n} + k \sum \limits_{i=0}^{n} b_{i}
\leq
\exp\left(
k \sum \limits_{i=0}^{n} \sigma_{i} \gamma_{i}
\right) \left(
k \sum \limits_{i=0}^{n} c_{i} + B
\right),
\quad \forall ~ n \geq 0.
\end{align*}
\end{theorem}

We choose $\bm{V}_{h}$ to be the $k_{1}+1$-th order polynomial space, $\bm{V}_{h}^{d}$ the $k_{2}$-th order Raviart-Thomas elements, $\bm{V}_{h}^{c}$ the $k_{3}$-th order N\'{e}d\'{e}lec element, and $Q_{h}$ the $k_{1}$-th order polynomial space.

Define $\bm{\xi} = ( \bm{u}, \bm{B}, \bm{j} )$, $\bm{\eta} = ( \bm{v}, \bm{C}, \bm{F} )$, and
\begin{align}
& a( \bm{\xi}, \bm{\xi}, \bm{\eta} )
= c( \bm{u}, \bm{u}, \bm{v} )
    + R_{e}^{-1} (\nabla \bm{u}, \nabla \bm{v})
    - s ( \bm{j} \times \bm{B},\bm{v} )
    + \alpha \left( \nabla\times ( \bm{j} - \bm{u} \times \bm{B} ), \bm{C} \right)
    \nonumber \\
    &\qquad\qquad
    + s ( \bm{j},\bm{F})
    - \alpha ( \bm{B}, \nabla\times \bm{F}),
    \quad \forall \bm{\xi}, ~ \bm{\eta} \in \bm{X},
    \label{eq:mhd_forma} \\
& b( \bm{\eta}, q ) = -( \nabla \cdot \bm{v}, q ),
\quad \forall \bm{\eta} \in \bm{X}, ~ q \in Q.
 \label{eq:mhd_formb}
\end{align}
Therefore, we can write the MHD system \eqref{eq:variational_form} in the form of a saddle point problem. That is, find $(\bm{\xi}, p) \in \bm{X} \times Q$ such that for any $(\bm{\eta}, q) \in \bm{X} \times Q$,
\begin{align}
\begin{cases}
& ( A \bm{\xi}_{t}, \bm{\eta} ) + a( \bm{\xi}, \bm{\xi}, \bm{\eta} ) + b( \bm{\eta}, p )
= \langle \bm{h},  \bm{\eta} \rangle, \\
& b( \bm{\xi}, q ) = 0,
\end{cases}
\label{eq:mhd_variational}
\end{align}
where $A = \mathrm{diag}(1, \alpha, 0)$, and $\bm{h} = ( \bm{f}, \bm{0}, \bm{0} )$. Additionally, we can write \eqref{eq:fully_picard} as: find $( \bm{\xi}_{h}^{n}, p_{h}^{n}  ) \in \bm{X}_{h} \times Q_{h}$ such that for any $( \bm{\eta}_{h}, q_{h} ) \in \bm{X}_{h} \times Q_{h}$,
\begin{align}
\begin{cases}
& ( A \bar{\partial}  \bm{\xi}_{h}^{n}, \bm{\eta}_{h} )
+ a( \bm{\xi}_{h}^{n}, { \bm{\xi} }_{h}^{n}, \bm{\eta}_{h} )
+ b( \bm{\eta}_{h}, { p }_{h}^{n} )
= \langle { \bm{h} }^{n}, \bm{\eta}_{h} \rangle, \\
& b( { \bm{\xi} }_{h}^{n}, q_{h} )
= 0.
\end{cases}
\label{eq:mhd_fem}
\end{align}

Before giving the detailed error estimates, we define the projections of $( \bm{\xi}^{n}, p^{n} )$ first. Assume that $( \widehat{ \bm{u} }_{h}^{n}, \widehat{ p }_{h}^{n} ) \in \bm{V}_{h} \times Q_{h}$ is the Stokes projection of $( \bm{u}^{n}, p^{n} ) \in \bm{V}^{0} \times Q$. That is, for any given $( \bm{u}^{n}, p^{n} ) \in \bm{V}^{0} \times Q$, we define $( \widehat{ \bm{u} }_{h}^{n}, \widehat{p}_{h}^{n} ) \in \bm{V}_{h} \times Q_{h}$ such that $( \rho_{ \bm{u} }^{n}, \rho_{p}^{n} )$ satisfies
\begin{align*}
\begin{cases}
& R_{e}^{-1} ( \nabla \rho_{ \bm{u} }^{n} , \nabla \bm{v}_{h} )
- ( \rho_{p}^{n} , \nabla \cdot \bm{v}_{h} )
= 0,
\quad \forall \bm{v}_{h} \in \bm{V}_{h}, \\
& ( \nabla \cdot \rho_{ \bm{u} }^{n}, q_{h} ) = 0,
\quad \forall q_{h} \in Q_{h}.
\end{cases}
\end{align*}
We choose $\widehat{ \bm{B} }_{h}^{n} \in \bm{V}_{h}^{d}$ as the $L^{2}$ projection of $\bm{B}^{n}$, and $\widehat{ \bm{E} }_{h}^{n} \in \bm{V}_{h}^{c}$ as the canonical interpolation of $\bm{E}^{n}$.

Define
\begin{align*}
& e_{ \bm{w} }^{n} = \bm{w}^{n} - \bm{w}_{h}^{n},
\quad \rho_{ \bm{w} }^{n} = \bm{w}^{n} - \widehat{ \bm{w} }_{ h }^{n},
\quad \theta_{ \bm{w} }^{n} = \widehat{ \bm{w} }_{ h }^{n} - \bm{w}_{h}^{n},
\quad
\bm{w} = \bm{u}, ~ \bm{B}, ~ \bm{E}, ~ p, ~\bm{j}.
\end{align*}
By the definitions of $\bm{j}^{n}$ and $\bm{j}_{h}^{n}$, we have
\begin{align*}
e_{\bm{j}}^{n}
& = \bm{E}^{n} + \bm{u}^{n} \times \bm{B}^{n}
- \left(
\bm{E}_{h}^{n} + \bm{u}_{h}^{n} \times \bm{B}_{h}^{n}
\right)
\\
& = e_{ \bm{E} }^{n} + \bm{u}^{n} \times \bm{B}^{n}
- \bm{u}_{h}^{n} \times \bm{B}_{h}^{n}
\\
& = e_{ \bm{E} }^{n} + \bm{u}^{n} \times e_{ \bm{B} }^{n}
+ e_{ \bm{u} }^{n} \times \bm{B}_{h}^{n}.
\end{align*}
Similarly, we define
\begin{align*}
&  \rho_{ \bm{j} }^{n}
= \rho_{ \bm{E} }^{n} + \bm{u}^{n} \times \rho_{ \bm{B} }^{n} + \rho_{ \bm{u} }^{n} \times \bm{B}_{h}^{n},
\quad
\theta_{ \bm{j} }^{n}
= \theta_{ \bm{E} }^{n} + \bm{u}^{n} \times \theta_{ \bm{B} }^{n} + \theta_{ \bm{u} }^{n} \times \bm{B}_{h}^{n}.
\end{align*}
It follows that $e_{ \bm{j} }^{n} = \rho_{\bm{j} }^{n} + \theta_{ \bm{j} }^{n}$. Moreover, we define
\begin{align*}
& e_{ \bm{\xi} }^{n} = \begin{pmatrix}
e_{ \bm{u} }^{n}, e_{ \bm{B} }^{n}, e_{ \bm{j} }^{n}
\end{pmatrix},
\quad
\rho_{ \bm{\xi} }^{n} = \begin{pmatrix}
\rho_{ \bm{u} }^{n}, \rho_{ \bm{B} }^{n}, \rho_{ \bm{j} }^{n}
\end{pmatrix},
\quad
\theta_{ \bm{\xi} }^{n} = \begin{pmatrix}
\theta_{ \bm{u} }^{n}, \theta_{ \bm{B} }^{n}, \theta_{ \bm{j} }^{n}
\end{pmatrix}.
\end{align*}
For simplicity, we denote
\begin{align*}
&\Vert\rho_{\bm \xi}^n\Vert_{L^2\times L^2 \times \mathrm{curl} }^2=\Vert \rho_{\bm u}^n\Vert^2 + \Vert\rho_{\bm B}^n \Vert^2 + \Vert \rho_{\bm E}^n\Vert_{\mathrm{curl}}^2,
\\
&\Vert\rho_{\bm \xi}^n\Vert_{H^1\times L^2 \times \mathrm{curl} }^2=\Vert \nabla\rho_{\bm u}^n\Vert^2 + \Vert\rho_{\bm B}^n \Vert^2 + \Vert \rho_{\bm E}^n\Vert_{\mathrm{curl}}^2,
\\
&\Vert A_1 (\bar{\partial}\bm{\xi}^n-\bm{\xi}_t^n)\Vert^2=\Vert \bar{\partial}\bm{u}^n-\bm{u}_t^n\Vert^2 + \Vert \bar{\partial}\bm{B}^n-\bm{B}_t^n\Vert^2,
\\
&\Vert A_1 \bar{\partial}\rho_{\bm{\xi}}^n\Vert^2=\Vert \bar{\partial}\rho_{\bm{u}}^n\Vert^2 + \Vert \bar{\partial}\rho_{\bm{B}}^n\Vert^2,
\\
&\rho_0= \max_{1\leq n\leq N}\left\{
\Vert A_1 (\bar{\partial}\bm{\xi}^n-\bm{\xi}_t^n)\Vert^2
+ \Vert A_1 \bar{\partial} \rho_{\bm{\xi}}^{n} \Vert^{2}
+ \Vert \rho_{\bm{\xi}}^{n} \Vert_{L^2\times L^2 \times \mathrm{curl}}^{2}
\right\},
\\
&\rho_1= \max_{1\leq n\leq N}\left\{
\Vert A_1 (\bar{\partial}\bm{\xi}^n-\bm{\xi}_t^n)\Vert^2
+ \Vert A_1 \bar{\partial} \rho_{\bm{\xi}}^{n} \Vert^{2}
+ \Vert \rho_{\bm{\xi}}^{n} \Vert_{H^1\times L^2 \times \mathrm{curl}}^{2}
\right\},
\end{align*}
where $A_1=diag(1,1,0).$
Noticing that
\begin{align*}
& c( { \bm{u} }^{n}, { \bm{u} }^{n}, \bm{v}_{h} )
- c ( { \bm{u} }^{n}_{h}, { \bm{u} }^{n}_{h}, \bm{v}_{h} )
=
c ( e_{ \bm{u} }^{n}, { \bm{u} }^{n}, \bm{v}_{h} )
+ c( { \bm{u} }^{n}_{h}, e_{ \bm{u} }^{n}, \bm{v}_{h} ),
\\
& ( { \bm{j} }^{n} \times { \bm{B} }^{n}, \bm{v}_{h} )
- ( { \bm{j} }^{n}_{h} \times { \bm{B} }^{n}_{h}, \bm{v}_{h} )
=
( { \bm{j} }^{n} \times e_{ \bm{B} }^{n}, \bm{v}_{h} )
+ ( e_{ \bm{j} }^{n} \times { \bm{B} }^{n}_{h}, \bm{v}_{h} ),
\end{align*}
we can rewrite the error equation as
\begin{align}
\begin{cases}
& ( A \bar{\partial} \theta_{ \bm{\xi} }^{n}, \bm{\eta}_{h} )
+ \widehat{a}( \bm{\xi}^{n}, \bm{\xi}_{h}^{n}, \theta_{ \bm{\xi} }^{n}, \bm{\eta}_{h} )
+ b( \bm{\eta}_{h}, \theta_{p}^{n} )
=
( A \bar{\partial} \bm{\xi}^{n} - A \bm{\xi}_{t}^{n}, \bm{\eta}_{h} )
\\
& \qquad
- ( A \bar{\partial} \rho_{ \bm{\xi} }^{n}, \bm{\eta}_{h} )
- \widehat{a}( \bm{\xi}^{n}, \bm{\xi}_{h}^{n}, \rho_{ \bm{\xi} }^{n}, \bm{\eta}_{h} )
- b( \bm{\eta}_{h}, \rho_{p}^{n} ),
\quad
\forall \bm{\eta}_{h} \in \bm{X}_{h},
 \\
& b ( \theta_{ \bm{\xi} }^{n}, q_{h} )
= -b ( \rho_{ \bm{\xi} }^{n}, q_{h} ),
\quad
\forall q_{h} \in Q_{h},
\end{cases}
\label{eq:mhd_erroreqn}
\end{align}
where
\begin{align*}
\widehat{a}( \bm{\xi}^{n}, \bm{\xi}_{h}^{n}, \theta_{ \bm{\xi} }^{n}, \bm{\eta}_{h} )
& =
c ( \theta_{ \bm{u} }^{n}, { \bm{u} }^{n}, \bm{v}_{h} )
+ c( { \bm{u} }^{n}_{h}, \theta_{ \bm{u} }^{n}, \bm{v}_{h} )
+ R_{e}^{-1} ( \nabla \theta_{ \bm{u} }^{n}, \nabla \bm{v}_{h} )
- s ( { \bm{j} }^{n} \times \theta_{ \bm{B} }^{n}, \bm{v}_{h} )
\\
& \qquad
- s ( \theta_{ \bm{j} }^{n} \times { \bm{B} }^{n}_{h}, \bm{v}_{h} )
+ \alpha \left(
\nabla \times ( \theta_{ \bm{j} }^{n}  - \bm{u}^{n} \times \theta_{ \bm{B} }^{n}
- \theta_{ \bm{u} }^{n} \times \bm{B}_{h}^{n} ),
\bm{C}_{h} \right)
\\
& \qquad
+ s ( \theta_{ \bm{j} }^{n}, \bm{F}_{h} )
- \alpha ( \theta_{ \bm{B} }^{n}, \nabla \times \bm{F}_{h} ).
\end{align*}

\subsection{Main results}
We summarize main results of this paper for error estimates of \eqref{eq:fully_picard} in the following theorem.
\begin{theorem}\label{thm:error_estimate_summary}
For any fixed time step $m$ such that $1\leq m \leq N$, if $\bm{\xi}^{m} = ( \bm{u}^{m}, \bm{B}^{m}, \bm{E}^{m} )$ is the solution to \eqref{eq:mhd_variational}, and $\bm{\xi}_{h}^{m} = ( \bm{u}^{m}_{h}, \bm{B}^{m}_{h}, \bm{E}^{m}_{h} )$ is the solution to \eqref{eq:fully_picard}, the following estimates hold:
\begin{enumerate}
\item There exists a constant $C$, which only depends on $\Vert \bm{u}^{n} \Vert_{0, \infty}$, $\Vert \nabla \bm{u}^{n} \Vert_{0, 3}$, $\Vert \bm{j}^{n} \Vert_{0, \infty}$ and the computation domain, such that
\begin{align}
 & \Vert \bm{u}^{m} -\bm{u}^{m}_{h} \Vert^{2}
+ \alpha \Vert \bm{B}^{m} - \bm{B}^{m}_h\Vert^{2}
+ R_{e}^{-1} \VERT \nabla \theta_{\bm{u}} \VERT^{2}_{m, 0}
+ s \VERT \bm{j} - \bm{j}_{h} \VERT_{m, 0}^{2}
\leq
C \rho_0,
\label{eq:estimate_xi_J}
\end{align}
when the time step size $k$ is sufficiently small. 
And there also holds
\begin{align}
& k \VERT p - p_{h} \VERT_{m,0}^{2}
\leq C\left(\max_{1\leq n \leq N} \Vert \rho_p^n\Vert^2+ \rho_0 + h^{-1} \rho_0^2\right).
\label{eq:error_p1}
\end{align}

\item  There exists a constant $C$, which only depends on $\Vert \bm{u}^{n} \Vert_{0, \infty}$, $\Vert \nabla \bm{u}^{n} \Vert_{0, 3}$, $\Vert \bm{B}^{n} \Vert_{0, \infty}$, $\Vert \bm{j}^{n} \Vert_{0, \infty}$ and the computation domain, such that 
\begin{align}
& \VERT \bm{E} - \bm{E}_{h} \VERT^{2}_{m, 0}
+ k \VERT \nabla \times \bm{E} - \nabla \times \bm{E}_{h} \VERT_{m, 0}^{2}
\leq
 C \rho_0 \left(1+h^{-1}\rho_0\right).
\label{eq:error_curlE}
\end{align}

\item There exists a constant $C$, which only depends on $\Vert \bm{u}^{n} \Vert_{0, \infty}$, $\Vert \nabla \bm{u}^{n} \Vert_{0, 3}$, $\Vert \bm{B}^{n} \Vert_{0, \infty}$, $\Vert \bm{j}^{n} \Vert_{0, \infty}$ and the computation domain, such that 
\begin{align}
& \VERT p - p_{h} \VERT_{m, 0}^{2}
\leq
C\left(\rho_0+\rho_0^2h^{-3}+\rho_1+\rho_1\rho_0 h^{-1}+\max_{1\leq n \leq N} \Vert \rho_p^n\Vert^2\right).
\label{eq:error_p2}
\end{align}

\end{enumerate}
\end{theorem}

By similar arguments similar to Theorem \ref{thm:error_estimate_summary}, we can get the error estimates of the Picard linearization scheme \eqref{prob:picard_linearization} as follows.
\begin{theorem}\label{thm:error estimate picard linearization}
For any fixed time step $m$ such that $1\leq m \leq N$, we have the following error estimates of \eqref{prob:picard_linearization}:
\begin{enumerate}
\item There exists a constant $C$, only depending on the exact solution, such that
\begin{align*}
 & \Vert \bm{u}^{m} -\bm{u}^{m}_{h} \Vert^{2}
+ \alpha \Vert \bm{B}^{m} - \bm{B}^{m}_h\Vert^{2}
+ R_{e}^{-1} \VERT \nabla \theta_{\bm{u}} \VERT^{2}_{m, 0}
+ s \VERT \bm{j} - \bm{j}_{h} \VERT_{m, 0}^{2}
\leq C\rho_0,
\end{align*}
when the time step size $k$ is sufficiently small.  And there also holds 
\begin{align*}
& k \VERT p - p_{h} \VERT_{m,0}^{2}
\leq
 C\left(\max_{1\leq n \leq N} \Vert \rho_p^n\Vert^2+ \rho_0 + h^{-1} \rho_0^2\right).
\end{align*}

\item
There exists a constant $C$ only depending on exact solution such that
\begin{align*}
& \VERT \bm{E} - \bm{E}_{h} \VERT^{2}_{m, 0}
+ k \VERT \nabla \times \bm{E} - \nabla \times \bm{E}_{h} \VERT_{m, 0}^{2}
\leq
C\rho_0\left(1+h^{-1}\rho_0\right),
\end{align*}
when the time step size $k$ is sufficiently small.

\item 
There exists a constant $C$ only depending on exact solution such that
\begin{align*}
& \VERT p - p_{h} \VERT_{m, 0}^{2}
\leq
C\left(\rho_0+\rho_0^2h^{-3}+\rho_1+\rho_1\rho_0 h^{-1}+\max_{1\leq n \leq N} \Vert \rho_p^n\Vert^2\right).
\end{align*}
when the time step size $k$ is sufficiently small.
\end{enumerate}
\end{theorem}

As the proof of the above theorem is similar to that of theorem \ref{thm:error_estimate_summary}, we omitted its details in this paper.

\subsection{Proof of Theorem \ref{thm:error_estimate_summary}}

The basic idea of our proof is to verify that $\widehat{a}( \bm{\xi}^{n}, \bm{\xi}_{h}^{n}, \theta_{ \bm{\xi} }^{n}, \bm{\eta}_{h} )$ satisfies the G\r{a}rding condition, and the truncation error is bounded by some norm of $\rho_{ \bm{\xi} }$. Then the conclusion follows by Gronwall's inequality.

First of all, we prove that $\widehat{a}( \bm{\xi}^{n}, \bm{\xi}_{h}^{n}, \theta_{ \bm{\xi} }^{n}, \bm{\eta}_{h} )$ satisfies the G\r{a}rding condition.

\begin{lemma}\label{lemma:garding_condition}
The sum of bilinear form $\widehat{a}( \bm{\xi}^{n}, \bm{\xi}_{h}^{n}, \theta_{ \bm{\xi} }^{n}, \bm{\eta}_{h} )$ satisfies the G\r{a}rding condition. That is, for any $\theta_{ \bm{\xi} }^{n} \in \bm{X}_{h}$, there exists $\bm{\eta}_{h} \in \bm{X}_{h}$, such that
\begin{align*}
& \widehat{a}( \bm{\xi}^{n}, \bm{\xi}_{h}^{n}, \theta_{ \bm{\xi} }^{n}, \bm{\eta}_{h} )
\geq
\beta_{1} \Vert \nabla \theta_{\bm{u}}^{n} \Vert^{2}
	+ \beta_{1} \Vert \theta_{\bm{j}}^{n} \Vert^{2}
	- \beta_{0} \Vert \theta_{\bm{u}}^{n} \Vert^{2}
	- \beta_{0} \Vert \theta_{\bm{B}}^{n} \Vert^{2},
\end{align*}
where $\beta_{0}$ and $\beta_{1}$ are positive constants that only depend on $\Vert \bm{u}^{n} \Vert_{0, \infty}$, $\Vert \bm{j}^{n} \Vert_{0, \infty}$, and the computation domain.
\end{lemma}

\begin{proof}

By definition, we know that ${ {\theta} }_{ \bm{\xi} }^{n} = \left( { \theta}_{ \bm{u} }^{n}, { \theta}_{ \bm{B} }^{n}, { \theta}_{ \bm{j} }^{n} \right)$. Since $\theta_{ \bm{j} }^{n}  - \bm{u}^{n} \times \theta_{ \bm{B} }^{n}
- \theta_{ \bm{u} }^{n} \times \bm{B}_{h}^{n}  = \theta_{ \bm{E} }^{n} \in \bm{V}_{h}^{c}$, we choose $\bm{\eta}_{h} = \left(
{ \theta}_{ \bm{u} }^{n}, { \theta}_{ \bm{B} }^{n},
\theta_{ \bm{j} }^{n}  - \bm{u}^{n} \times \theta_{ \bm{B} }^{n}
- \theta_{ \bm{u} }^{n} \times \bm{B}_{h}^{n}
\right)$. Noticing that $c( { \bm{u} }^{n}_{h}, {\theta}_{ \bm{u} }^{n}, {\theta}_{ \bm{u} }^{n} ) = 0$, we get
	\begin{align*}
	& \widehat{a}( \bm{\xi}^{n}, \bm{\xi}_{h}^{n}, \theta_{ \bm{\xi} }^{n}, \bm{\eta}_{h} )
	\\
	= &
	c ( {\theta}_{ \bm{u} }^{n}, { \bm{u} }^{n}, {\theta}_{ \bm{u} }^{n} )
	+ c( { \bm{u} }^{n}_{h}, {\theta}_{ \bm{u} }^{n}, {\theta}_{ \bm{u} }^{n} )
	+ R_{e}^{-1} ( \nabla {\theta}_{ \bm{u} }^{n}, \nabla {\theta}_{ \bm{u} }^{n} )
	- s ( { \bm{j} }^{n} \times {\theta}_{ \bm{B} }^{n}, {\theta}_{ \bm{u} }^{n} )
	 \\
	& \quad
	- s ( {\theta}_{ \bm{j} }^{n} \times { \bm{B} }^{n}_{h}, {\theta}_{ \bm{u} }^{n} )
	+ s ( {\theta}_{ \bm{j} }^{n}, \theta_{ \bm{j} }^{n}  - \bm{u}^{n} \times \theta_{ \bm{B} }^{n}
	- \theta_{ \bm{u} }^{n} \times \bm{B}_{h}^{n} )
	\\
	= &
	c ( {\theta}_{ \bm{u} }^{n}, { \bm{u} }^{n}, {\theta}_{ \bm{u} }^{n} )
	+ R_{e}^{-1} ( \nabla {\theta}_{ \bm{u} }^{n}, \nabla {\theta}_{ \bm{u} }^{n} )
	- s ( { \bm{j} }^{n} \times {\theta}_{ \bm{B} }^{n}, {\theta}_{ \bm{u} }^{n} )
	+ s ( {\theta}_{ \bm{j} }^{n}, {\theta}_{ \bm{j} }^{n} )
	- s ( {\theta}_{ \bm{j} }^{n} , \bm{u}^{n} \times \theta_{ \bm{B} }^{n} ).
	\end{align*}
	By the Cauchy-Schwarz inequality, we obtain
	\begin{align*}
	& \lvert ( { \bm{j} }^{n} \times {\theta}_{ \bm{B} }^{n}, {\theta}_{ \bm{u} }^{n} ) \rvert
	\leq
	C \Vert {\theta}_{ \bm{B} }^{n} \Vert^{2}
	+ C \Vert {\theta}_{ \bm{u} }^{n} \Vert^{2},
	\\
	& \lvert ( {\theta}_{ \bm{j} }^{n} , \bm{u}^{n} \times \theta_{ \bm{B} }^{n} ) \rvert
	\leq
	C \Vert {\theta}_{ \bm{j} }^{n} \Vert^{2}
	+ C \Vert {\theta}_{ \bm{B} }^{n} \Vert^{2}.
	\end{align*}
	Therefore, the conclusion holds.

\end{proof}

\begin{lemma}\label{lemma:truncation_error}
The truncation error $\widehat{a}( \bm{\xi}^{n}, \bm{\xi}_{h}^{n}, \rho_{ \bm{\xi} }^{n}, \bm{\eta}_{h} )
+ b( \bm{\eta}_{h}, \rho_{p}^{n} )$ is bounded. That is, there exists a constant $C$ such that
\begin{align*}
\widehat{a}( \bm{\xi}^{n}, \bm{\xi}_{h}^{n}, \rho_{ \bm{\xi} }^{n}, \bm{\eta}_{h} )
+ b( \bm{\eta}_{h}, \rho_{p}^{n} )
\leq  &
C \left[
\Vert \nabla \theta_{\bm{u}}^{n} \Vert
\Vert \bm{v}_{h} \Vert
+
\left(
\Vert \theta_{\bm{u}}^{n} \Vert
+ \Vert \rho_{\bm{u}}^{n} \Vert
\right)
\Vert \nabla \bm{v}_{h} \Vert
+ \Vert \rho_{\bm{B}}^{n} \Vert
\Vert \bm{v}_{h} \Vert
\right.
\\
&
\left.
+ \Vert \nabla \times \rho_{\bm{E}}^{n} \Vert
\Vert \bm{C}_{h} \Vert
+
\Vert \rho_{ \bm{j} }^{n} \Vert
\left(
\Vert \bm{F}_{h} + \bm{v}_{h} \times { \bm{B} }^{n}_{h} + \bm{u}^{n} \times \bm{C}_{h} \Vert
+ \Vert \bm{C}_{h} \Vert \right)
\right],
\end{align*}
where $C$ only depends on $\Vert \bm{u}^{n} \Vert_{0, \infty}$, $\Vert \nabla \bm{u}^{n} \Vert_{0, 3}$, $\Vert \bm{j}^{n} \Vert_{0, \infty}$, and the computation domain.
\end{lemma}

\begin{proof}
By the definition of projections, we know that
\begin{align*}
& R_{e}^{-1} ( \nabla \rho_{\bm{u}}^{n}, \nabla \bm{v}_{h} )
- ( \rho_{p}^{n}, \nabla \cdot  \bm{v}_{h} ) = 0,
\quad \forall \bm{v}_{h} \in \bm{V}_{h}, \\
& ( \rho_{\bm{B}}^{n}, \nabla \times \bm{F}_{h} ) = 0,
\quad
\forall \bm{F}_{h} \in \bm{V}^{c}_{h}.
\end{align*}
Therefore, by definition of $\rho_{\bm{j}}^{n}$, we have
\begin{align*}
\widehat{a}( \bm{\xi}^{n}, \bm{\xi}_{h}^{n}, \rho_{ \bm{\xi} }^{n}, \bm{\eta}_{h} )
+ b( \bm{\eta}_{h}, \rho_{p}^{n} )
& =
c ( \rho_{ \bm{u} }^{n}, { \bm{u} }^{n}, \bm{v}_{h} )
+ c( { \bm{u} }^{n}_{h}, \rho_{ \bm{u} }^{n}, \bm{v}_{h} )
- s ( { \bm{j} }^{n} \times \rho_{ \bm{B} }^{n}, \bm{v}_{h} )
\\
& \qquad
- s ( \rho_{ \bm{j} }^{n} \times { \bm{B} }^{n}_{h}, \bm{v}_{h} )
+ \alpha( \nabla \times \rho_{ \bm{E} }^{n}, \bm{C}_{h} )
+ s ( \rho_{ \bm{j} }^{n}, \bm{F}_{h} ).
\end{align*}
To prove the boundedness of truncation error, we only need to verify that all the terms in the above expression are bounded.
\begin{align*}
\lvert c ( \rho_{ \bm{u} }^{n}, { \bm{u} }^{n}, \bm{v}_{h} ) \rvert
& \leq
C \Vert \rho_{ \bm{u} }^{n} \Vert
\Vert \nabla { \bm{u} }^{n} \Vert_{0,3}
\Vert \bm{v}_{h} \Vert_{0,6}
+
C \Vert \rho_{ \bm{u} }^{n} \Vert
\Vert \nabla \bm{v}_{h} \Vert
\Vert { \bm{u} }^{n} \Vert_{0, \infty}
\leq
C \Vert \rho_{ \bm{u} }^{n} \Vert
\Vert \nabla \bm{v}_{h} \Vert.
\end{align*}
Since $ c( \bm{u}_{h}^{n}, \rho_{ \bm{u} }^{n}, \bm{v}_{h} )
=
- c( \theta_{\bm{u}}^{n}, \rho_{ \bm{u} }^{n}, \bm{v}_{h} )
+  c( \widehat{ \bm{u} }_{h}^{n}, \rho_{ \bm{u} }^{n}, \bm{v}_{h} ) $ and
\begin{align*}
& \lvert c( \theta_{\bm{u}}^{n}, \rho_{ \bm{u} }^{n}, \bm{v}_{h} ) \rvert
\leq
C
\Vert \nabla \theta_{\bm{u}}^{n} \Vert
\Vert \nabla \rho_{ \bm{u} }^{n} \Vert_{0,3}
\Vert \bm{v}_{h} \Vert
+ C
\Vert \theta_{\bm{u}}^{n} \Vert
\Vert \rho_{\bm{u}}^{n} \Vert_{0, \infty}
\Vert \nabla \bm{v}_{h} \Vert,
\\
& \lvert c( \widehat{ \bm{u} }_{h}^{n}, \rho_{ \bm{u} }^{n}, \bm{v}_{h} ) \rvert
= \lvert c ( \widehat{ \bm{u} }_{h}^{n}, \bm{v}_{h}, \rho_{\bm{u}}^{n} ) \rvert
=
\lvert ( \widehat{ \bm{u} }_{h}^{n} \cdot \nabla \bm{v}_{h}, \rho_{\bm{u}}^{n} )
+ \dfrac{1}{2}
 ( (\nabla \cdot \widehat{ \bm{u} }_{h}^{n}) \bm{v}_{h}, \rho_{\bm{u}}^{n} ) \rvert
\\
& \qquad \qquad \qquad \quad
\leq
\Vert \widehat{ \bm{u} }_{h}^{n} \Vert_{0, \infty}
\Vert \nabla \bm{v}_{h} \Vert
\Vert \rho_{\bm{u}}^{n} \Vert
+ C \Vert \nabla \widehat{ \bm{u} }_{h}^{n} \Vert_{0,3}
\Vert \nabla \bm{v}_{h} \Vert
\Vert \rho_{\bm{u}}^{n} \Vert,
\end{align*}
we get
\begin{align*}
& \lvert c( \bm{u}_{h}^{n}, \rho_{ \bm{u} }^{n}, \bm{v}_{h} ) \rvert
\leq
C \Vert \nabla \theta_{\bm{u}}^{n} \Vert
\Vert \bm{v}_{h} \Vert
+ C
\left(
\Vert \theta_{\bm{u}}^{n} \Vert
+ \Vert \rho_{\bm{u}}^{n} \Vert
\right)
\Vert \nabla \bm{v}_{h} \Vert.
\end{align*}
Moreover,
\begin{align*}
& \lvert ( \bm{j}^{n} \times \rho_{\bm{B}}^{n}, \bm{v}_{h} ) \rvert
\leq
\Vert \bm{j}^{n} \Vert_{0, \infty}
\Vert \rho_{\bm{B}}^{n} \Vert
\Vert \bm{v}_{h} \Vert
\leq
C \Vert \rho_{\bm{B}}^{n} \Vert
\Vert \bm{v}_{h} \Vert, \\
& \lvert ( \nabla \times \rho_{\bm{E}}^{n}, \bm{C}_{h} ) \rvert
\leq \Vert \nabla \times \rho_{\bm{E}}^{n} \Vert \Vert \bm{C}_{h} \Vert,
\end{align*}
and
\begin{align*}
s ( \rho_{ \bm{j} }^{n} , \bm{v}_{h} \times { \bm{B} }^{n}_{h} )
+ s ( \rho_{ \bm{j} }^{n}, \bm{F}_{h} )
= &
s (\rho_{ \bm{j} }^{n}, \bm{F}_{h} + \bm{v}_{h} \times { \bm{B} }^{n}_{h} + \bm{u}^{n} \times \bm{C}_{h} )
- s( \rho_{ \bm{j} }^{n}, \bm{u}^{n} \times \bm{C}_{h} )
\\
\leq &
C \Vert \rho_{ \bm{j} }^{n} \Vert
\left(
\Vert \bm{F}_{h} + \bm{v}_{h} \times { \bm{B} }^{n}_{h} + \bm{u}^{n} \times \bm{C}_{h} \Vert
+ \Vert \bm{C}_{h} \Vert
\right).
\end{align*}
The conclusion follows.

\end{proof}

%

\noindent\emph{Proof of \eqref{eq:estimate_xi_J} }.
 In equation \eqref{eq:mhd_erroreqn}, taking $\bm{\eta}_{h} = \left( \theta_{\bm{u}}^{n}, \theta_{\bm{B}}^{n},
\theta_{ \bm{j} }^{n}  - \bm{u}^{n} \times \theta_{ \bm{B} }^{n}
- \theta_{ \bm{u} }^{n} \times \bm{B}_{h}^{n}
\right) $, we get $b( \bm{\eta}_{h}, \theta_{p}^{n} ) = 0$. Therefore,
\begin{align*}
( A \bar{\partial} \theta_{\bm{\xi}}^{n}, \bm{\eta}_{h} )
+ \widehat{a}( \bm{\xi}^{n}, \bm{\xi}_{h}^{n}, \theta_{ \bm{\xi} }^{n}, \bm{\eta}_{h} )
= &
( A \bar{\partial} \bm{\xi}^{n} - A \bm{\xi}_{t}^{n}, \bm{\eta}_{h} )
- ( A \bar{\partial} \rho_{\bm{\xi}}^{n}, \bm{\eta}_{h} )
- \widehat{a}( \bm{\xi}^{n}, \bm{\xi}_{h}^{n}, \rho_{ \bm{\xi} }^{n}, \bm{\eta}_{h} )
- b( \bm{\eta}_{h}, \rho_{p}^{n} ).
\end{align*}
By the conclusion of Lemma \ref{lemma:garding_condition}, we have
\begin{align*}
& \widehat{a}( \bm{\xi}^{n}, \bm{\xi}_{h}^{n}, \theta_{ \bm{\xi} }^{n}, \bm{\eta}_{h} )
\geq
\beta_{1} \Vert \nabla \theta_{\bm{u}}^{n} \Vert^{2}
+ \beta_{1} \Vert \theta_{\bm{j}}^{n} \Vert^{2}
- \beta_{0} \Vert \theta_{\bm{u}}^{n} \Vert^{2}
- \beta_{0} \Vert \theta_{\bm{B}}^{n} \Vert^{2}.
\end{align*}
For the right-hand side, we have
\begin{align*}
& ( A \bar{\partial} \bm{\xi}^{n} - A \bm{\xi}_{t}^{n}, \bm{\eta}_{h} )
= ( \bar{\partial} \bm{u}^{n} - \bm{u}_{t}^{n}, \theta_{\bm{u}}^{n} )
+ \alpha ( \bar{\partial} \bm{B}^{n} - \bm{B}_{t}^{n}, \theta_{\bm{B}}^{n} )
\\
& \qquad \qquad \qquad \qquad
\leq
\Vert \bar{\partial} \bm{u}^{n} - \bm{u}_{t}^{n} \Vert^{2}
+ \Vert \theta_{\bm{u}}^{n} \Vert^{2}
+ \alpha \Vert \bar{\partial} \bm{B}^{n} - \bm{B}_{t}^{n} \Vert^{2}
+ \alpha \Vert \theta_{\bm{B}}^{n} \Vert^{2} ,
\\
& \lvert ( A \bar{\partial} \rho_{\bm{\xi}}^{n}, \bm{\eta}_{h} ) \rvert
\leq
\lvert ( \bar{\partial} \rho_{\bm{u}}^{n}, \theta_{\bm{u}}^{n} )  \rvert
+ \lvert \alpha ( \bar{\partial} \rho_{\bm{B}}^{n}, \theta_{\bm{B}}^{n} ) \rvert
\leq
C \left(
\Vert \bar{\partial} \rho_{\bm{u}}^{n} \Vert^{2}
+ \Vert \theta_{\bm{u}}^{n} \Vert^{2}
+ \Vert  \bar{\partial} \rho_{\bm{B}}^{n} \Vert^{2}
+ \Vert \theta_{\bm{B}}^{n} \Vert^{2}
\right).
\end{align*}
And by the conclusion of Lemma \ref{lemma:truncation_error},
\begin{align*}
\lvert
\widehat{a}( \bm{\xi}^{n}, \bm{\xi}_{h}^{n}, \theta_{ \bm{\xi} }^{n}, \bm{\eta}_{h} )
+ b( \bm{\eta}_{h}, \rho_{p}^{n} ) \rvert
& \leq
C \left[
\Vert \nabla \theta_{\bm{u}}^{n} \Vert
\Vert \theta_{\bm{u}}^{n} \Vert
+
\left(
\Vert \theta_{\bm{u}}^{n} \Vert
+ \Vert \rho_{\bm{u}}^{n} \Vert
\right)
\Vert \nabla \theta_{\bm{u}}^{n} \Vert
+ \Vert \rho_{\bm{B}}^{n} \Vert
\Vert \theta_{\bm{u}}^{n} \Vert
\right.
\\
& \qquad
\left.
+ \Vert \nabla \times \rho_{\bm{E}}^{n} \Vert
\Vert \theta_{\bm{B}}^{n} \Vert
+ \Vert \rho_{ \bm{j} }^{n} \Vert
\left(
\Vert \theta_{\bm{j}}^{n} \Vert
+ \Vert \theta_{\bm{B}}^{n} \Vert \right)
\right].
\end{align*}

Noticing that
\begin{align*}
& ( \bar{\partial} \theta_{\bm{u}}^{n}, \theta_{\bm{u}}^{n} )
= \dfrac{1}{2k} \left(
\Vert \theta_{\bm{u}}^{n} \Vert^{2}
- \Vert \theta_{\bm{u}}^{n-1} \Vert^{2}
+ \Vert \theta_{\bm{u}}^{n} - \theta_{\bm{u}}^{n-1} \Vert^{2}
\right), \\
& ( \bar{\partial} \theta_{\bm{B}}^{n}, \theta_{\bm{B}}^{n} )
= \dfrac{1}{2k} \left(
\Vert \theta_{\bm{B}}^{n} \Vert^{2}
- \Vert \theta_{\bm{B}}^{n-1} \Vert^{2}
+ \Vert \theta_{\bm{B}}^{n} - \theta_{\bm{B}}^{n-1} \Vert^{2}
\right),
\end{align*}
we obtain
\begin{align*}
& \dfrac{1}{2k} \left(
\Vert \theta_{\bm{u}}^{n} \Vert^{2}
- \Vert \theta_{\bm{u}}^{n-1} \Vert^{2}
\right)
+ \dfrac{k}{2} \Vert \bar{\partial} \theta_{\bm{u}}^{n} \Vert^{2}
+ \dfrac{1}{2k} \left(
\Vert \theta_{\bm{B}}^{n} \Vert^{2}
- \Vert \theta_{\bm{B}}^{n-1} \Vert^{2}
\right)
+ \dfrac{k}{2} \Vert \bar{\partial} \theta_{\bm{B}}^{n} \Vert^{2}
\\
& \qquad
+ \beta_{1} \Vert \nabla \theta_{\bm{u}}^{n} \Vert^{2}
+ \beta_{1} \Vert \theta_{\bm{j}}^{n} \Vert^{2}
- \beta_{0} \Vert \theta_{\bm{u}}^{n} \Vert^{2}
- \beta_{0} \Vert \theta_{\bm{B}}^{n} \Vert^{2}
\\
\leq &
\Vert \bar{\partial} \bm{u}^{n} - \bm{u}_{t}^{n} \Vert^{2}
+ \alpha \Vert \bar{\partial} \bm{B}^{n} - \bm{B}_{t}^{n} \Vert^{2}
+ C \Vert \bar{\partial} \rho_{\bm{u}}^{n} \Vert^{2}
+ C \Vert \bar{\partial} \rho_{\bm{B}}^{n} \Vert^{2}
+ \dfrac{1}{2} \beta_{1} \Vert \nabla \theta_{\bm{u}}^{n} \Vert^{2}
+ \dfrac{1}{2} \beta_{1} \Vert \theta_{\bm{j}}^{n} \Vert^{2}
\\
& \qquad
+ C \Vert \theta_{\bm{u}}^{n} \Vert^{2}
+ C \Vert \theta_{\bm{B}}^{n} \Vert^{2}
+ C \Vert \rho_{\bm{u}}^{n} \Vert^{2}
+ C \Vert \rho_{\bm{B}}^{n} \Vert^{2}
+ C \Vert \rho_{\bm{j}}^{n} \Vert^{2}
+ C \Vert \nabla \times \rho_{\bm{E}}^{n} \Vert^{2}.
\end{align*}
Kicking back $\Vert \nabla \theta_{\bm{u}}^{n} \Vert^{2}$ and $\Vert \theta_{\bm{j}}^{n} \Vert^{2}$, we get
\begin{align*}
& \dfrac{1}{2k} \left(
\Vert \theta_{\bm{u}}^{n} \Vert^{2}
- \Vert \theta_{\bm{u}}^{n-1} \Vert^{2}
\right)
+ \dfrac{k}{2} \Vert \bar{\partial} \theta_{\bm{u}}^{n} \Vert^{2}
+ \dfrac{1}{2k} \left(
\Vert \theta_{\bm{B}}^{n} \Vert^{2}
- \Vert \theta_{\bm{B}}^{n-1} \Vert^{2}
\right)
+ \dfrac{k}{2} \Vert \bar{\partial} \theta_{\bm{B}}^{n} \Vert^{2}
\\
& \qquad
+ \dfrac{1}{2} \beta_{1} \Vert \nabla \theta_{\bm{u}}^{n} \Vert^{2}
+ \dfrac{1}{2} \beta_{1} \Vert \theta_{\bm{j}}^{n} \Vert^{2}
\leq
C \Vert \theta_{\bm{u}}^{n} \Vert^{2}
+ C \Vert \theta_{\bm{B}}^{n} \Vert^{2}
+ C R_{n},
\end{align*}
where
$ R_{n}
= \Vert \bar{\partial} \bm{u}^{n} - \bm{u}_{t}^{n} \Vert^{2}
+ \alpha \Vert \bar{\partial} \bm{B}^{n} - \bm{B}_{t}^{n} \Vert^{2}
+ \Vert \bar{\partial} \rho_{\bm{u}}^{n} \Vert^{2}
+ \Vert \bar{\partial} \rho_{\bm{B}}^{n} \Vert^{2}
+ \Vert \rho_{\bm{u}}^{n} \Vert^{2}
+ \Vert \rho_{\bm{B}}^{n} \Vert^{2}
+ \Vert \rho_{\bm{j}}^{n} \Vert^{2}
+ \Vert \nabla \times \rho_{\bm{E}}^{n} \Vert^{2}$.
Noticing that
\begin{align*}
\Vert \rho_{\bm{j}}^{n} \Vert^{2}
& \leq
\left(
\Vert \rho_{\bm{E}}^{n} \Vert
+ \Vert \bm{u}^{n} \times \rho_{\bm{B}}^{n} \Vert
+ \Vert \rho_{\bm{u}}^{n} \times \bm{B}_{h}^{n} \Vert
\right)^{2}
\\
& \leq
\left(
\Vert \rho_{\bm{E}}^{n} \Vert
+ \Vert \bm{u}^{n} \times \rho_{\bm{B}}^{n} \Vert
+ \Vert \rho_{\bm{u}}^{n} \times \theta_{\bm{B}}^{n} \Vert
+ \Vert \rho_{\bm{u}}^{n} \times \widehat{ \bm{B} }_{h}^{n} \Vert
\right)^{2}
\\
& \leq
C \Vert \rho_{\bm{E}}^{n} \Vert^{2}
+ C \Vert \rho_{\bm{B}}^{n} \Vert^{2}
+ C \Vert \theta_{\bm{B}}^{n} \Vert^{2}
+ C \Vert \rho_{\bm{u}}^{n} \Vert^{2},
\end{align*}
we know that
\begin{align*}
&  \left(
\Vert \theta_{\bm{u}}^{n} \Vert^{2}
- \Vert \theta_{\bm{u}}^{n-1} \Vert^{2}
\right)
+  k^2 \Vert \bar{\partial} \theta_{\bm{u}}^{n} \Vert^{2}
+  \left(
\Vert \theta_{\bm{B}}^{n} \Vert^{2}
- \Vert \theta_{\bm{B}}^{n-1} \Vert^{2}
\right)
+ k^2 \Vert \bar{\partial} \theta_{\bm{B}}^{n} \Vert^{2}
+ k \beta_{1} \Vert \nabla \theta_{\bm{u}}^{n} \Vert^{2}
+ k  \beta_{1} \Vert \theta_{\bm{j}}^{n} \Vert^{2}
\\
&\quad \leq
C k \Vert \theta_{\bm{u}}^{n} \Vert^{2}
+ C k \Vert \theta_{\bm{B}}^{n} \Vert^{2}
+ C k \rho_0.
\end{align*}
By Gronwall's inequality, we have
\begin{align*}
&
\Vert \theta_{\bm{u}}^{m} \Vert^{2}
+  k \VERT \bar{\partial} \theta_{\bm{u}} \VERT_{m,0}^{2}
+\Vert \theta_{\bm{B}}^{m} \Vert^{2}
+ k \VERT \bar{\partial} \theta_{\bm{B}} \VERT_{m,0}^{2}
+  \beta_{1} \VERT \nabla \theta_{\bm{u}} \VERT_{m,0}^{2}
+   \beta_{1} \VERT \theta_{\bm{j}} \VERT_{m,0}^{2}
\\
&\quad \leq
C \rho_0.
\end{align*}
Then, by triangle inequality, we readily obtain \eqref{eq:estimate_xi_J} and
\begin{align}
& k \VERT \bar{\partial} \theta_{\bm{u}} \VERT^{2}_{m,0}
\leq
C \rho_0,
\label{eq:error_partial_thetau} \\
& k \VERT \bar{\partial} \theta_{\bm{B}} \VERT^{2}_{m,0}
\leq
C\rho_0.
\label{eq:error_partial_thetaB}
\end{align}
\hfill \ensuremath{\Box}


\begin{remark}
Estimates \eqref{eq:error_partial_thetau} and \eqref{eq:error_partial_thetaB} are useful for analyzing the error of the electric field and the pressure.
\end{remark}

The above theorem gives the $L^{2}$ estimate of the volume current density $\bm{j}$ and the electric field $\bm{E}$. Since the MHD system \eqref{eq:variational_form} is well-posed \cite{Hu.K;Ma.Y;Xu.J.2014a} with respect to the norm $\Vert \cdot \Vert_{ \bm{X} }$, we need to further estimate $\Vert \nabla \times e_{\bm{E}} \Vert$.


\noindent\emph{Proof of \eqref{eq:error_curlE} }.
By the definition of $\theta_{\bm{j}}^{n}$, we get
\begin{align*}
& \theta_{\bm{E}}^{n} =
\theta_{\bm{j}}^{n} - \bm{u}^{n} \times \theta_{\bm{B}}^{n}
- \theta_{\bm{u}}^{n} \times \bm{B}_{h}^{n}
= \theta_{\bm{j}}^{n} - \bm{u}^{n} \times \theta_{\bm{B}}^{n}
+ \theta_{\bm{u}}^{n} \times \theta_{\bm{B}}^{n}
- \theta_{\bm{u}}^{n} \times \widehat{ \bm{B} }_{h}^{n}.
\end{align*}
By the Cauchy-Schwarz inequality and inverse inequality, we have
\begin{align*}
& \Vert \bm{u}^{n} \times \theta_{\bm{B}}^{n} \Vert^{2}
\leq C \Vert \theta_{\bm{B}}^{n} \Vert^{2},
\\
& \Vert \theta_{\bm{u}}^{n} \times \widehat{ \bm{B} }_{h}^{n} \Vert^{2}
\leq
\Vert \widehat{ \bm{B} }_{h}^{n} \Vert_{0, \infty}^{2}
\Vert \theta_{\bm{u}}^{n} \Vert^{2}
\leq
C \Vert \theta_{\bm{u}}^{n} \Vert^{2},
\\
& \Vert \theta_{\bm{u}}^{n} \times \theta_{\bm{B}}^{n} \Vert^{2}
\leq
\Vert \theta_{\bm{u}}^{n} \Vert_{0, \infty}^{2}
\Vert \theta_{\bm{B}}^{n} \Vert^{2}
\leq
C h^{-1}
\Vert \theta_{\bm{u}}^{n} \Vert^{2}_{0,6}
\Vert \theta_{\bm{B}}^{n} \Vert^{2}
\leq
C h^{-1}
\Vert \nabla \theta_{\bm{u}}^{n} \Vert^{2}
\Vert \theta_{\bm{B}}^{n} \Vert^{2}.
\end{align*}
The last estimate follows from the inverse inequality. By estimate \eqref{eq:estimate_xi_J}, we have
\begin{align*}
& \VERT \theta_{\bm{E}} \VERT^{2}_{m,0}
\leq
C \VERT \theta_{\bm{j}} \VERT^{2}_{m,0}
+C \VERT \theta_{\bm{B}} \VERT^{2}_{m,0}
+ C \VERT \theta_{\bm{u}} \VERT^{2}_{m,0}
+ C h^{-1} \max_{1\leq n\leq N}\Vert \theta_{\bm{B}}^n\Vert^2
\VERT \nabla \theta_{\bm{u}} \VERT^{2}_{m,0}
\\
&\quad \leq C \rho_0(1+h^{-1} \rho_0).
\end{align*}

%
Next, we will estimate the second term. By the error equation \eqref{eq:mhd_erroreqn}, we get
\begin{align*}
& ( \bar{\partial} \theta_{\bm{B}}^{n}, \bm{C}_{h} )
+ ( \nabla \times \theta_{\bm{E}}^{n}, \bm{C}_{h} )
=
( \bar{\partial} \bm{B}^{n} - \bm{B}_{t}^{n}, \bm{C}_{h} )
- ( \bar{\partial} \rho_{\bm{B}}^{n}, \bm{C}_{h} )
- (\nabla \times \rho_{\bm{E}}^{n}, \bm{C}_{h} ).
\end{align*}
Taking $\bm{C}_{h} = \nabla \times \theta_{\bm{E}}^{n}$, we obtain
\begin{align*}
\Vert \nabla \times \theta_{\bm{E}}^{n} \Vert^{2}
& =
- ( \bar{\partial} \theta_{\bm{B}}^{n}, \nabla \times \theta_{\bm{E}}^{n} )
+ ( \bar{\partial} \bm{B}^{n} - \bm{B}_{t}^{n}, \nabla \times \theta_{\bm{E}}^{n} )
- ( \bar{\partial} \rho_{\bm{B}}^{n}, \nabla \times \theta_{\bm{E}}^{n} )
- (\nabla \times \rho_{\bm{E}}^{n}, \nabla \times \theta_{\bm{E}}^{n} ).
\end{align*}
By the Cauchy-Schwarz inequality, we get
\begin{align*}
& \Vert \nabla \times \theta_{\bm{E}}^{n} \Vert^{2}
\leq
2 \Vert \bar{\partial} \theta_{\bm{B}}^{n} \Vert^{2}
+ 2 \Vert \bar{\partial} \bm{B}^{n} - \bm{B}_{t}^{n} \Vert^{2}
+ 2 \Vert \bar{\partial} \rho_{\bm{B}}^{n} \Vert^{2}
+ 2 \Vert \rho_{\bm{B}}^{n} \Vert^{2}
+ 2 \Vert \nabla \times \rho_{\bm{E}}^{n} \Vert^{2}
+ \dfrac{1}{2} \Vert \nabla \times \theta_{\bm{E}}^{n} \Vert^{2}.
\end{align*}
Therefore, using (\ref{eq:error_partial_thetaB}), we have
\begin{align*}
& k \VERT \nabla \times \theta_{\bm{E}} \VERT^{2}_{m,0}
\leq C \rho_0.
\end{align*}
The conclusion \eqref{eq:error_curlE} follows by the existing estimate of each term.

\hfill \ensuremath{\Box}

Now, we get the $H^{1}$ estimate of the velocity, $H(\mathrm{curl})$ estimate of the electric field $\bm{E}$, and $L^{2}$ estimate of the magnetic field $\bm{B}$. Next, we estimate the $L^{2}$ error of the pressure $p$. Since the MHD system \eqref{eq:variational_form} is a saddle-point system with $b(\cdot, \cdot)$ a bilinear form, we can apply the standard error estimating technique.


\noindent\emph{Proof of \eqref{eq:error_p1} }.
Because $b ( \bm{\eta}_{h}, \theta_{p}^{n} ) = ( \nabla \cdot \bm{v}_{h}, \theta_{p}^{n} )$ for any $\bm{\eta}_{h} = ( \bm{v}_{h}, \bm{C}_{h}, \bm{F}_{h} )$, we only need to consider the first error equation in \eqref{eq:mhd_erroreqn}. So the error equation is
\begin{align*}
( \nabla \cdot \bm{v}_{h}, \theta_{p}^{n} )
= &
- ( \bar{\partial} \theta_{ \bm{u} }^{n}, \bm{v}_{h} )
- c ( \theta_{ \bm{u} }^{n}, { \bm{u} }^{n}, \bm{v}_{h} )
- R_{e}^{-1} ( \nabla \theta_{ \bm{u} }^{n}, \nabla \bm{v}_{h} )
+ s ( { \bm{j} }^{n} \times \theta_{ \bm{B} }^{n}, \bm{v}_{h} )
\\
& \qquad
- c( { \bm{u} }^{n}_{h}, \theta_{ \bm{u} }^{n}, \bm{v}_{h} )
+ s ( \theta_{ \bm{j} }^{n} \times { \bm{B} }^{n}_{h}, \bm{v}_{h} )
- ( \bar{\partial} \rho_{ \bm{u} }^{n}, \bm{v}_{h} )
+ ( \bar{\partial} \bm{u}^{n} - \bm{u}_{t}^{n}, \bm{v}_{h} )
\\
& \qquad
-c ( \rho_{ \bm{u} }^{n}, { \bm{u} }^{n}, \bm{v}_{h} )
- c( { \bm{u} }^{n}_{h}, \rho_{ \bm{u} }^{n}, \bm{v}_{h} )
+ s ( { \bm{j} }^{n} \times \rho_{ \bm{B} }^{n}, \bm{v}_{h} )
+ s ( \rho_{ \bm{j} }^{n} \times { \bm{B} }^{n}_{h}, \bm{v}_{h} ).
\end{align*}
By the Cauchy-Schwarz inequality,
\begin{equation*}
\begin{aligned}[c]
& \lvert ( \bar{\partial} \theta_{ \bm{u} }^{n}, \bm{v}_{h} ) \rvert
\leq
\Vert \bar{\partial} \theta_{ \bm{u} }^{n} \Vert
\Vert \bm{v}_{h} \Vert, \\
& \lvert c ( \theta_{ \bm{u} }^{n}, { \bm{u} }^{n}, \bm{v}_{h} ) \rvert
\leq
C \Vert \theta_{ \bm{u} }^{n} \Vert
\Vert \nabla \bm{v}_{h} \Vert, \\
& \lvert ( \nabla \theta_{ \bm{u} }^{n}, \nabla \bm{v}_{h} ) \rvert
\leq
\Vert \nabla \theta_{ \bm{u} }^{n} \Vert
\Vert  \nabla \bm{v}_{h} \Vert , \\
& \lvert ( { \bm{j} }^{n} \times \theta_{ \bm{B} }^{n}, \bm{v}_{h} ) \rvert
\leq
C \Vert \theta_{ \bm{B} }^{n} \Vert
\Vert \bm{v}_{h} \Vert,
\end{aligned}
\qquad
\begin{aligned}
& \lvert ( \bar{\partial} \bm{u}^{n} - \bm{u}_{t}^{n}, \bm{v}_{h} ) \rvert
\leq
\Vert \bar{\partial} \bm{u}^{n} - \bm{u}_{t}^{n} \Vert
\Vert \bm{v}_{h} \Vert,
\\
& \lvert ( \bar{\partial} \rho_{ \bm{u} }^{n}, \bm{v}_{h} ) \rvert
\leq
\Vert \bar{\partial} \rho_{ \bm{u} }^{n} \Vert
\Vert \bm{v}_{h} \Vert, \\
& \lvert c ( \rho_{ \bm{u} }^{n}, { \bm{u} }^{n}, \bm{v}_{h} ) \rvert
\leq
C \Vert \rho_{ \bm{u} }^{n} \Vert
\Vert \nabla \bm{v}_{h} \Vert, \\
& \lvert ( { \bm{j} }^{n} \times \rho_{ \bm{B} }^{n}, \bm{v}_{h} ) \rvert
\leq
C \Vert \rho_{ \bm{B} }^{n} \Vert
\Vert \bm{v}_{h} \Vert,
\end{aligned}
\end{equation*}
and by inverse inequality,
\begin{align*}
\lvert c( { \bm{u} }^{n}_{h}, \theta_{ \bm{u} }^{n}, \bm{v}_{h} ) \rvert
& =
\lvert c( \widehat{ \bm{u} }_{h}^{n} - \theta_{ \bm{u} }^{n}, \theta_{ \bm{u} }^{n}, \bm{v}_{h} ) \rvert
\\
& \leq
\left(
\Vert \widehat{ \bm{u} }_{h}^{n} \Vert_{0, \infty}
+ \Vert \theta_{ \bm{u} }^{n} \Vert_{0, \infty}
\right)
\Vert \theta_{ \bm{u} }^{n} \Vert
\Vert \nabla\bm{v}_{h} \Vert
+ C\left(\Vert \nabla \widehat{ \bm{u} }_{h}^{n} \Vert_{0,3}
+ \Vert \nabla \theta_{ \bm{u} }^{n} \Vert_{0,3}\right)\Vert \theta_{ \bm{u} }^{n} \Vert
\Vert \nabla \bm{v}_{h} \Vert
\\
& \leq
C \left( 1 +
h^{-1/2} \Vert \theta_{ \bm{u} }^{n} \Vert_{0,6}
\right)
\Vert \theta_{ \bm{u} }^{n} \Vert
\Vert \nabla \bm{v}_{h} \Vert
+C \left( 1 +
h^{-1/2} \Vert \nabla\theta_{ \bm{u} }^{n} \Vert
\right)
\Vert \theta_{ \bm{u} }^{n} \Vert
\Vert \nabla \bm{v}_{h} \Vert
\\
& \leq
C \left( 1 +
h^{-1/2} \Vert \nabla \theta_{ \bm{u} }^{n} \Vert
\right)
 \Vert \theta_{ \bm{u} }^{n} \Vert
\Vert \nabla \bm{v}_{h} \Vert,
\\
\lvert ( \theta_{ \bm{j} }^{n} \times { \bm{B} }^{n}_{h}, \bm{v}_{h} ) \rvert
& \leq
\lvert ( \theta_{ \bm{j} }^{n} \times \theta_{ \bm{B} }^{n}, \bm{v}_{h} ) \rvert
+ \lvert ( \theta_{ \bm{j} }^{n} \times \widehat{ \bm{B} }_{h}^{n}, \bm{v}_{h} ) \rvert
\\
& \leq
C \Vert \theta_{ \bm{j} }^{n} \Vert
\Vert \theta_{ \bm{B} }^{n} \Vert
\Vert \bm{v}_{h} \Vert_{0, \infty}
+ C \Vert \theta_{ \bm{j} }^{n} \Vert
\Vert \widehat{ \bm{B} }_{h}^{n} \Vert_{0, 3}
\Vert \bm{v}_{h} \Vert_{0,6}
\\
& \leq
C h^{-1/2} \Vert \theta_{ \bm{j} }^{n} \Vert
\Vert \theta_{ \bm{B} }^{n} \Vert
\Vert \bm{v}_{h} \Vert_{0, 6}
+ C \Vert \theta_{ \bm{j} }^{n} \Vert
\Vert \nabla \bm{v}_{h} \Vert
\\
& \leq
C h^{-1/2} \Vert \theta_{ \bm{j} }^{n} \Vert
\Vert \theta_{ \bm{B} }^{n} \Vert
\Vert \nabla \bm{v}_{h} \Vert
+ C \Vert \theta_{ \bm{j} }^{n} \Vert
\Vert \nabla \bm{v}_{h} \Vert.
\end{align*}
By similar argument, we have
\begin{align*}
& \lvert ( \rho_{ \bm{j} }^{n} \times { \bm{B} }^{n}_{h}, \bm{v}_{h} ) \rvert
\leq
C h^{-1/2} \Vert \rho_{ \bm{j} }^{n} \Vert
\Vert \theta_{ \bm{B} }^{n} \Vert
\Vert \nabla \bm{v}_{h} \Vert
+ C \Vert \rho_{ \bm{j} }^{n} \Vert
\Vert \nabla \bm{v}_{h} \Vert.
\end{align*}
And by the conclusion of Lemma \ref{lemma:truncation_error},
\begin{align*}
& \lvert c( \bm{u}_{h}^{n}, \rho_{ \bm{u} }^{n}, \bm{v}_{h} ) \rvert
\leq
C \Vert \nabla \theta_{\bm{u}}^{n} \Vert
\Vert \nabla \bm{v}_{h} \Vert
+ C
 \Vert \rho_{\bm{u}}^{n} \Vert
\Vert \nabla \bm{v}_{h} \Vert.
\end{align*}
Therefore, by the inf-sup condition of $b(\cdot, \cdot)$, proven in \cite{Boffi.D;Brezzi.F;Fortin.M.2013a}, we get
\begin{align}
\label{eq:estimate_p1}
\zeta \Vert \theta_{p}^{n} \Vert
\leq &
C \left(
\Vert \bar{\partial} \theta_{\bm{u}}^{n} \Vert
+ \Vert \theta_{\bm{u}}^{n} \Vert
+ \Vert \nabla \theta_{\bm{u}}^{n} \Vert
+ h^{-1/2} \Vert \nabla \theta_{\bm{u}}^{n} \Vert
\Vert \theta_{\bm{u}}^{n} \Vert
+ \Vert \theta_{\bm{B}}^{n} \Vert
+ \Vert \theta_{\bm{j}}^{n} \Vert+ h^{-1/2} \Vert \theta_{\bm{j}}^{n} \Vert
\Vert \theta_{\bm{B}}^{n} \Vert
\right.
\\
& \quad
\left.
+ \Vert \bar{\partial} \rho_{\bm{u}}^{n} \Vert
+ \Vert \rho_{\bm{u}}^{n} \Vert
+ \Vert \rho_{\bm{B}}^{n} \Vert
+ h^{-1/2} \Vert \rho_{\bm{j}}^{n} \Vert
\Vert \theta_{\bm{B}}^{n} \Vert
+ \Vert \rho_{\bm{j}}^{n} \Vert+\Vert \bar{\partial} \bm{u}^n- \bm{u}^n_t\Vert
\right).
\nonumber
\end{align}
By estimates \eqref{eq:estimate_xi_J} and \eqref{eq:error_partial_thetau}, we have
\begin{align*}
k\VERT \theta_p\VERT_{m,0}^2\leq C \left(\rho_0 + h^{-1}\rho_0^2\right).
\end{align*}
And, by triangle inequality, we obtain \eqref{eq:error_p1}.

\hfill \ensuremath{\Box}

Based on the analysis of \eqref{eq:error_p1}, the estimate of $\VERT \bar{\partial} \theta_{\bm{u}} \VERT^{2}_{m,0}$ determines the estimated convergence rate of the pressure. So an improved result leads to a better $L^{2}$ estimate of the pressure. Knowing that the MHD system \eqref{eq:variational_form} is a coupled system of Navier-Stokes equation and Maxwell's equation, such improvement is achievable. We summarize this estimate in the following lemma.

\begin{lemma}\label{lemma:partial_thetau}
For any time step $1 \leq m \leq N$, we have
\begin{align}
\VERT \bar{\partial} \theta_{ \bm{u} }^{n} \VERT_{m,0}^{2}
+ R_{e}^{-1}
\Vert \nabla \theta_{ \bm{u} }^{m} \Vert^{2}
\leq
C\left(\rho_0 + h^{-3}\rho_0^2+\rho_1 + h^{-1}\rho_1\rho_0\right),
\end{align}
the $C$ is a constant only depending on $\Vert \bm{u}^{n} \Vert_{0, \infty}$, $\Vert \nabla \bm{u}^{n} \Vert_{0, 3}$, $\Vert \bm{B}^{n} \Vert_{0, \infty}$, $\Vert \bm{j}^{n} \Vert_{0, \infty}$, and the computation domain.
\end{lemma}

\begin{proof}

Since we focus on the $L^{2}$ estimate of $\bar{\partial} \theta_{\bm{u}}^{n}$, we only need to consider the first error equation in \eqref{eq:mhd_erroreqn}. Namely,
\begin{align*}
& ( \bar{\partial} \theta_{ \bm{u} }^{n}, \bm{v}_{h} )
+ R_{e}^{-1} ( \nabla \theta_{ \bm{u} }^{n}, \nabla \bm{v}_{h} )
- ( \nabla \cdot \bm{v}_{h}, \theta_{p}^{n} )
\\
= &
- c ( \theta_{ \bm{u} }^{n}, { \bm{u} }^{n}, \bm{v}_{h} )
+ s ( { \bm{j} }^{n} \times \theta_{ \bm{B} }^{n}, \bm{v}_{h} )
- c( { \bm{u} }^{n}_{h}, \theta_{ \bm{u} }^{n}, \bm{v}_{h} )
+ s ( \theta_{ \bm{j} }^{n} \times { \bm{B} }^{n}_{h}, \bm{v}_{h} )
\\
& \qquad
- ( \bar{\partial} \rho_{ \bm{u} }^{n}, \bm{v}_{h} )
+ ( \bar{\partial} \bm{u}^{n} - \bm{u}_{t}^{n}, \bm{v}_{h} )
-c ( \rho_{ \bm{u} }^{n}, { \bm{u} }^{n}, \bm{v}_{h} )
- c( { \bm{u} }^{n}_{h}, \rho_{ \bm{u} }^{n}, \bm{v}_{h} )
\\
& \qquad
+ s ( { \bm{j} }^{n} \times \rho_{ \bm{B} }^{n}, \bm{v}_{h} )
+ s ( \rho_{ \bm{j} }^{n} \times { \bm{B} }^{n}_{h}, \bm{v}_{h} ).
\end{align*}
When $\bm{v}_{h} = \bar{\partial} \theta_{ \bm{u} }^{n}$, $( \nabla \cdot \bm{v}_{h}, \theta_{p}^{n} ) = 0$. Therefore,
\begin{align*}
& \Vert \bar{\partial} \theta_{ \bm{u} }^{n} \Vert^{2}
+ {1 \over 2k} R_{e}^{-1} \left(
\Vert \nabla \theta_{ \bm{u} }^{n} \Vert^{2}
- \Vert \nabla \theta_{\bm{u}}^{n-1} \Vert^{2}
\right)
\\
\leq &
- c ( \theta_{ \bm{u} }^{n}, { \bm{u} }^{n}, \bar{\partial} \theta_{ \bm{u} }^{n} )
+ s ( { \bm{j} }^{n} \times \theta_{ \bm{B} }^{n}, \bar{\partial} \theta_{ \bm{u} }^{n} )
- c( { \bm{u} }^{n}_{h}, \theta_{ \bm{u} }^{n}, \bar{\partial} \theta_{ \bm{u} }^{n} )
+ s ( \theta_{ \bm{j} }^{n} \times { \bm{B} }^{n}_{h}, \bar{\partial} \theta_{ \bm{u} }^{n} )
\\
& \qquad
- ( \bar{\partial} \rho_{ \bm{u} }^{n}, \bar{\partial} \theta_{ \bm{u} }^{n} )
+ ( \bar{\partial} \bm{u}^{n} - \bm{u}_{t}^{n}, \bar{\partial} \theta_{ \bm{u} }^{n} )
-c ( \rho_{ \bm{u} }^{n}, { \bm{u} }^{n}, \bar{\partial} \theta_{ \bm{u} }^{n} )
- c( { \bm{u} }^{n}_{h}, \rho_{ \bm{u} }^{n}, \bar{\partial} \theta_{ \bm{u} }^{n} )
\\
& \qquad
+ s ( { \bm{j} }^{n} \times \rho_{ \bm{B} }^{n}, \bar{\partial} \theta_{ \bm{u} }^{n} )
+ s ( \rho_{ \bm{j} }^{n} \times { \bm{B} }^{n}_{h}, \bar{\partial} \theta_{ \bm{u} }^{n} ).
\end{align*}
By the Cauchy-Schwarz inequality, we have
\begin{align*}
& \lvert c ( \theta_{ \bm{u} }^{n}, { \bm{u} }^{n}, \bar{\partial} \theta_{ \bm{u} }^{n} ) \rvert
\leq
\lvert (  \theta_{ \bm{u} }^{n} \cdot \nabla { \bm{u} }^{n}, \bar{\partial} \theta_{ \bm{u} }^{n} ) \rvert
+ \dfrac{1}{2}
\lvert \left( ( \nabla \cdot \theta_{ \bm{u} }^{n} ) \bm{u}^{n}, \bar{\partial} \theta_{ \bm{u} }^{n} \right) \rvert
\\
& \qquad \qquad \qquad \quad
\leq
\Vert \theta_{ \bm{u} }^{n} \Vert_{0,6}
\Vert \nabla { \bm{u} }^{n} \Vert_{0, 3}
\Vert \bar{\partial} \theta_{ \bm{u} }^{n} \Vert
+ C
\Vert \nabla \theta_{ \bm{u} }^{n} \Vert
\Vert \bm{u}^{n} \Vert_{0,\infty}
\Vert \bar{\partial} \theta_{ \bm{u} }^{n} \Vert
\\
& \qquad \qquad \qquad \quad
\leq
C \Vert \nabla \theta_{ \bm{u} }^{n} \Vert
\Vert \bar{\partial} \theta_{ \bm{u} }^{n} \Vert,
\\
& \lvert ( { \bm{j} }^{n} \times \theta_{ \bm{B} }^{n}, \bar{\partial} \theta_{ \bm{u} }^{n} ) \rvert
\leq
\Vert { \bm{j} }^{n} \Vert_{0, \infty}
\Vert \theta_{ \bm{B} }^{n} \Vert
\Vert \bar{\partial} \theta_{ \bm{u} }^{n} \Vert
\leq
C \Vert \theta_{ \bm{B} }^{n} \Vert
\Vert \bar{\partial} \theta_{ \bm{u} }^{n} \Vert,
\end{align*}
and by triangle inequality,
\begin{align*}
& \lvert c( { \bm{u} }^{n}_{h}, \theta_{ \bm{u} }^{n}, \bar{\partial} \theta_{ \bm{u} }^{n} ) \rvert
\leq
\lvert c( \theta_{ \bm{u} }^{n}, \theta_{ \bm{u} }^{n}, \bar{\partial} \theta_{ \bm{u} }^{n} ) \rvert
+ \lvert c( \widehat{ \bm{u} }_{h}^{n}, \theta_{ \bm{u} }^{n}, \bar{\partial} \theta_{ \bm{u} }^{n} ) \rvert,
\\
& \lvert ( \theta_{ \bm{j} }^{n} \times { \bm{B} }^{n}_{h}, \bar{\partial} \theta_{ \bm{u} }^{n} ) \rvert
\leq
\lvert ( \theta_{ \bm{j} }^{n} \times \theta_{ \bm{B} }^{n}, \bar{\partial} \theta_{ \bm{u} }^{n} ) \rvert
+ \lvert ( \theta_{ \bm{j} }^{n} \times \widehat{ \bm{B} }_{h}^{n}, \bar{\partial} \theta_{ \bm{u} }^{n} ) \rvert.
\end{align*}
By the properties of the trilinear form $c(\cdot, \cdot, \cdot)$ and the inverse inequality,
\begin{align*}
\lvert c( \theta_{ \bm{u} }^{n}, \theta_{ \bm{u} }^{n}, \bar{\partial} \theta_{ \bm{u} }^{n} ) \rvert
& \leq
\lvert ( \theta_{ \bm{u} }^{n} \cdot \nabla \theta_{ \bm{u} }^{n},  \bar{\partial} \theta_{ \bm{u} }^{n} ) \rvert
+ \lvert \left( ( \nabla \cdot \theta_{ \bm{u} }^{n} ) \theta_{ \bm{u} }^{n},  \bar{\partial} \theta_{ \bm{u} }^{n} \right) \rvert
\\
& \leq
\Vert \theta_{ \bm{u} }^{n} \Vert_{0,\infty}
\Vert \nabla \theta_{ \bm{u} }^{n} \Vert
\Vert \bar{\partial} \theta_{ \bm{u} }^{n} \Vert
+ C
\Vert \nabla \theta_{ \bm{u} }^{n} \Vert
\Vert \theta_{ \bm{u} }^{n} \Vert_{0, \infty}
\Vert \bar{\partial} \theta_{ \bm{u} }^{n} \Vert
\\
& \leq
C h^{-3/2} \Vert \nabla \theta_{ \bm{u} }^{n} \Vert
\Vert \theta_{\bm u}^n\Vert
\Vert \bar{\partial} \theta_{ \bm{u} }^{n} \Vert,
\\
\lvert c( \widehat{ \bm{u} }_{h}^{n}, \theta_{ \bm{u} }^{n}, \bar{\partial} \theta_{ \bm{u} }^{n} ) \rvert
& \leq
\rvert ( \widehat{ \bm{u} }_{h}^{n} \cdot \nabla \theta_{ \bm{u} }^{n}, \bar{\partial} \theta_{ \bm{u} }^{n} ) \rvert
+ \lvert \left( (\nabla \cdot \widehat{ \bm{u} }_{h}^{n}) \theta_{ \bm{u} }^{n}, \bar{\partial} \theta_{ \bm{u} }^{n} \right)\rvert
\\
& \leq
\Vert \widehat{ \bm{u} }_{h}^{n} \Vert_{0,\infty}
\Vert \nabla \theta_{ \bm{u} }^{n} \Vert
\Vert \bar{\partial} \theta_{ \bm{u} }^{n} \Vert
+
\Vert \nabla \cdot \widehat{ \bm{u} }_{h}^{n} \Vert_{0,3}
\Vert \theta_{ \bm{u} }^{n} \Vert_{0,6}
\Vert \bar{\partial} \theta_{ \bm{u} }^{n} \Vert
\\
& \leq
C \Vert \nabla \theta_{ \bm{u} }^{n} \Vert
\Vert \bar{\partial} \theta_{ \bm{u} }^{n} \Vert.
\end{align*}
And by the similar argument,
\begin{align*}
& \lvert ( \theta_{ \bm{j} }^{n} \times \theta_{ \bm{B} }^{n}, \bar{\partial} \theta_{ \bm{u} }^{n} ) \rvert
\leq
\Vert \theta_{ \bm{j} }^{n} \Vert
\Vert \theta_{ \bm{B} }^{n} \Vert
\Vert \bar{\partial} \theta_{ \bm{u} }^{n} \Vert_{0, \infty}
\leq
C h^{-3/2} \Vert \theta_{ \bm{j} }^{n} \Vert
\Vert \theta_{ \bm{B} }^{n} \Vert
\Vert \bar{\partial} \theta_{ \bm{u} }^{n} \Vert,
\\
& \lvert ( \theta_{ \bm{j} }^{n} \times \widehat{ \bm{B} }_{h}^{n}, \bar{\partial} \theta_{ \bm{u} }^{n} ) \rvert
\leq
\Vert \theta_{ \bm{j} }^{n} \Vert
\Vert \widehat{ \bm{B} }_{h}^{n} \Vert_{0, \infty}
\Vert \bar{\partial} \theta_{ \bm{u} }^{n} \Vert
\leq
C \Vert \theta_{ \bm{j} }^{n} \Vert
\Vert \bar{\partial} \theta_{ \bm{u} }^{n} \Vert.
\end{align*}
Moreover, the above estimates still hold if we use $\rho_{\bm{u}}^{n}$ instead of $\theta_{\bm{u}}^{n}$. Kicking back $\Vert \bar{\partial} \theta_{\bm{u}}^{n} \Vert^{2}$, we get
\begin{align*}
& k \Vert \bar{\partial} \theta_{ \bm{u} }^{n} \Vert^{2}
+ R_{e}^{-1} \left(
\Vert \nabla \theta_{ \bm{u} }^{n} \Vert^{2}
- \Vert \nabla \theta_{\bm{u}}^{n-1} \Vert^{2}
\right)
\\
\leq &
C k \left(
\Vert \nabla \theta_{\bm{u}}^{n} \Vert^{2}
+ \Vert \theta_{\bm{B}}^{n} \Vert^{2}
+ \Vert \theta_{\bm{j}}^{n} \Vert^{2}
+ h^{-3} \Vert \nabla \theta_{\bm{u}}^{n} \Vert^{2} \Vert \theta_{\bm{u}}^{n} \Vert^2
+ h^{-3} \Vert \theta_{\bm{j}}^{n} \Vert^{2}
\Vert \theta_{\bm{B}}^{n} \Vert^{2}
+ \Vert \bar{\partial} \bm{u}^{n} - \bm{u}_{t}^{n} \Vert^{2}
\right.
\\
& \qquad
\left.
+ \Vert \bar{\partial} \rho_{\bm{u}}^{n} \Vert^{2}
+ \Vert \nabla \rho_{\bm{u}}^{n} \Vert^{2}
+ \Vert \rho_{\bm{B}}^{n} \Vert^{2}
+ \Vert \rho_{\bm{j}}^{n} \Vert^{2}
+ h^{-1} \Vert \nabla \rho_{\bm{u}}^{n} \Vert^{2}
\Vert \nabla \theta_{\bm{u}}^{n} \Vert^{2}
+ h^{-3} \Vert \rho_{\bm{j}}^{n} \Vert^{2}
\Vert \theta_{\bm{B}}^{n} \Vert^{2}
\right).
\end{align*}
Summing both sides from $1$ to $m$, we reach 
\begin{align*}
&  \VERT \bar{\partial} \theta_{ \bm{u} }^{n} \VERT_{m,0}^{2}
+ R_{e}^{-1}
\Vert \nabla \theta_{ \bm{u} }^{m} \Vert^{2}
\leq
C\left(\rho_0+h^{-3} \rho_0^2+\rho_1 +h^{-1} \rho_1\rho_0\right).
\end{align*}
\end{proof}

With the improved estimate of $\Vert \bar{\partial} \theta_{\bm{u}}^{n} \Vert$, we can reach a better estimate of the pressure.

\noindent\emph{Proof of \eqref{eq:error_p2}.}
The proof is identical to that of \eqref{eq:error_p1}, except that we use the estimate of $\Vert \bar{\partial} \theta_{\bm{u}}^{n} \Vert$ in Lemma \ref{lemma:partial_thetau} instead of \eqref{eq:error_partial_thetau} after obtaining \eqref{eq:estimate_p1}.

\hfill \ensuremath{\Box}

 We assume the regularity on the real solution $( \bm{u}, \bm{B}, \bm{E}, p )$ to the problem \eqref{eq:variational_form} is
$\bm{u} \in H^{ s_{1} }(\Omega)$, $\bm{B} \in H^{ s_{2} }(\mathrm{div}, \Omega)$, $ \bm{E} \in H^{ s_{3} }(\mathrm{curl}, \Omega), p \in H^{ s_{4} }(\Omega)$, $\bm u_t \in H^{s_1}(\Omega), \bm B_t \in H^{s_2}(\Omega), p_t \in H^{s_4}(\Omega)$, $\bm u_{tt} \in L^2(\Omega), \bm{B}_{tt} \in L^2(\Omega)$.  Usually, we assume that $s_{1} \geq 3/2$, $s_{2} > 1/2$, $s_{3} > 1/2$, and $s_{4} \geq 1/2$.
By the error estimate of the saddle point problem, we know that on a convex domain
\begin{align*}
& \Vert \rho_{\bm{u}}^{n} \Vert
+ h \left(
\Vert \nabla \rho_{\bm{u}}^{n} \Vert
+ \Vert \rho_{p}^{n} \Vert
\right)
\leq
C h^{r_{1}+1} \left(
\Vert \bm{u}^{n} \Vert_{r_{1}+1,2}
+ \Vert p^n \Vert_{r_{1},2}
\right),
\end{align*}
and
\begin{align*}
& \Vert \rho_{\bm{B}}^{n} \Vert_{\mathrm{div}}
\leq
C h^{r_{2}} \Vert \bm{B}^{n} \Vert_{r_{2}, \mathrm {div}},
\quad
\Vert \rho_{\bm{E}}^{n} \Vert_{\mathrm{curl}}
\leq C h^{r_{3}} \Vert \bm{E}^{n} \Vert_{r_{3}, \mathrm{curl}},
\quad
r_{2} > {1 \over 2},
~ r_{3} > {1 \over 2}.
\end{align*}
where $r_{1} = \min \left\{ s_{1}-1, s_{4}, k_{1} \right\}$,
$r_{2} = \min \left\{ s_{2}, k_{2}+1 \right\}$, $r_{3} = \min \left\{ s_{3}, k_{3}+1 \right\}$. Detailed proof of the above property is in chapter 5 of \cite{Monk.P.2003a}. Therefore, under the above assumptions, we can have the error orders of $\rho_0$ and $\rho_1$:
$$\rho_0 \leq C(k^2+ h^{2\hat{r}}),\quad  \rho_1 \leq C(k^2+ h^{2r}).$$
Where, $\hat{r}=\min\{r_1+1, r_2, r_3\},  r=\min\{r_1, r_2, r_3\}$.
Thus, based on Theorem \ref{thm:error_estimate_summary}, we obtain the error orders of Algorithm \ref{prob:picard}.
\begin{theorem}
For any fixed time step $m$ such that $1\leq m \leq N$, if $\bm{\xi}^{m} = ( \bm{u}^{m}, \bm{B}^{m}, \bm{E}^{m} )$ is the solution to \eqref{eq:mhd_variational}, and $\bm{\xi}_{h}^{m} = ( \bm{u}^{m}_{h}, \bm{B}^{m}_{h}, \bm{E}^{m}_{h} )$ is the solution to \eqref{eq:fully_picard}, the following estimates hold:
\begin{enumerate}
\item There exists a constant $C$, which only depends on $\Vert \bm{u}^{n} \Vert_{0, \infty}$, $\Vert \nabla \bm{u}^{n} \Vert_{0, 3}$, $\Vert \bm{j}^{n} \Vert_{0, \infty}$ and the computation domain, such that
\begin{align*}
 & \Vert \bm{u}^{m} -\bm{u}^{m}_{h} \Vert^{2}
+ \alpha \Vert \bm{B}^{m} - \bm{B}^{m}_h\Vert^{2}
+ R_{e}^{-1} \VERT \nabla \theta_{\bm{u}} \VERT^{2}_{m, 0}
+ s \VERT \bm{j} - \bm{j}_{h} \VERT_{m, 0}^{2}
\leq
C \left( k^2 + h^{2 \hat{r}} \right),
\end{align*}
when the time step size $k$ is sufficiently small.  And when $k < h^{1/2}$, we have
\begin{align*}
& k \VERT p - p_{h} \VERT_{m,0}^{2}
\leq C \left( k^{2} + h^{2 \hat{r} } \right).
\end{align*}

\item  There exists a constant $C$, which only depends on $\Vert \bm{u}^{n} \Vert_{0, \infty}$, $\Vert \nabla \bm{u}^{n} \Vert_{0, 3}$, $\Vert \bm{B}^{n} \Vert_{0, \infty}$, $\Vert \bm{j}^{n} \Vert_{0, \infty}$ and the computation domain, such that when $k\leq h^{1/2}$,
\begin{align*}
& \VERT \bm{E} - \bm{E}_{h} \VERT^{2}_{m, 0}
+ k \VERT \nabla \times \bm{E} - \nabla \times \bm{E}_{h} \VERT_{m, 0}^{2}
\leq
 C \left( k^{2} + h^{2 \hat{r}} \right).
\end{align*}

\item There exists a constant $C$, which only depends on $\Vert \bm{u}^{n} \Vert_{0, \infty}$, $\Vert \nabla \bm{u}^{n} \Vert_{0, 3}$, $\Vert \bm{B}^{n} \Vert_{0, \infty}$, $\Vert \bm{j}^{n} \Vert_{0, \infty}$ and the computation domain, such that when $k \leq h^{3/2}$,
\begin{align*}
& \VERT p - p_{h} \VERT_{m, 0}^{2}
\leq
C \left( k^{2} + h^{ 2 r } \right).
\end{align*}

\end{enumerate}
\end{theorem}
Similarly, we have the error orders about Algorithm \ref{prob:picard_linearization} based on Theorem \ref{thm:error estimate picard linearization}.
\begin{theorem}
For any fixed time step $m$ such that $1\leq m \leq N$, we have the following error estimates of \eqref{prob:picard_linearization}:
\begin{enumerate}
\item There exists a constant $C$, only depending on the exact solution, such that
\begin{align*}
 & \Vert \bm{u}^{m} -\bm{u}^{m}_{h} \Vert^{2}
+ \alpha \Vert \bm{B}^{m} - \bm{B}^{m}_h\Vert^{2}
+ R_{e}^{-1} \VERT \nabla \theta_{\bm{u}} \VERT^{2}_{m, 0}
+ s \VERT \bm{j} - \bm{j}_{h} \VERT_{m, 0}^{2}
\leq
C \left( k^2 + h^{2 \hat{r}} \right),
\end{align*}
when the time step size $k$ is sufficiently small.   And if $k < h^{1/2}$, we have
\begin{align*}
& k \VERT p - p_{h} \VERT_{m,0}^{2}
\leq
C \left( k^{2} + h^{2 \hat{r} } \right).
\end{align*}

\item
If  $k\leq h^{1/2}$,  we have
there exists a constant $C$ only depending on exact solution such that
\begin{align*}
& \VERT \bm{E} - \bm{E}_{h} \VERT^{2}_{m, 0}
+ k \VERT \nabla \times \bm{E} - \nabla \times \bm{E}_{h} \VERT_{m, 0}^{2}
\leq
C \left( k^{2} + h^{2 \hat{r}} \right),
\end{align*}
when the time step size $k$ is sufficiently small.

\item If $k \leq h^{3/2}$, we have
there exists a constant $C$ only depending on exact solution such that
\begin{align*}
& \VERT p - p_{h} \VERT_{m, 0}^{2}
\leq
C \left( k^{2} + h^{ 2 r } \right),
\end{align*}
when the time step size $k$ is sufficiently small.
\end{enumerate}
\end{theorem}


%
\section{Numerical results}\label{sec:numer_tests}
In this section, we present the results of the numerical experiment. The chosen exact solution is
\begin{align*}
& \bm{u} = \begin{pmatrix}
e^{t} \cos y \\ 0 \\ 0
\end{pmatrix},
\quad \bm{E} = \begin{pmatrix}
0 \\ \cos x \\ 0
\end{pmatrix},
\quad \bm{B} = \begin{pmatrix}
0 \\ 0 \\ e^{t} \cos x
\end{pmatrix},
\quad p = -x \cos y.
\end{align*}
And we compute the right-hand side based on the exact solution. To measure the error, we use the norms identical to our analysis. Namely, if $( \bm{u}^{m}, \bm{B}^{m}, \bm{E}^{m}, p^{m} )$ is the exact solution at time $t_{m}$, and $( \bm{u}^{m}_{h}, \bm{B}^{m}_{h}, \bm{E}^{m}_{h}, p^{m}_{h} )$ is the corresponding numeric solution, the errors in Figure \ref{fig:picard_convergence} are computed by
\begin{align*}
& \Vert \bm{u}^{m} - \bm{u}^{m}_{h} \Vert_{\ast}^{2}
= \Vert \bm{u}^{m} - \bm{u}^{m}_{h} \Vert^{2}
+ \VERT \nabla \bm{u} - \nabla \bm{u}_{h} \VERT_{m,0}^{2},
\\
& \Vert \bm{B}^{m} - \bm{B}^{m}_{h} \Vert_{\ast}^{2}
= \Vert \bm{B}^{m} - \bm{B}^{m}_{h} \Vert^{2},
\\
& \Vert \bm{E}^{m} - \bm{E}^{m}_{h} \Vert_{\ast}^{2}
= \VERT \bm{E} - \bm{E}_{h} \VERT_{m,0}^{2}
+ k \VERT \nabla \times \bm{E} - \nabla \times \bm{E}_{h} \VERT_{m,0}^{2},
\\
& \Vert p^{m} - p^{m}_{h} \Vert^{2}_{\ast}
= \VERT p - p_{h} \VERT_{m,0}^{2}.
\end{align*}
We use $P_{2}-P_{1}$ to discretize the velocity and pressure pair, the lowest-order Ravi\'{a}rt-Thomas element to discretize the magnetic field, and the lowest-order N\'{e}delec edge element to discretize the electric field. Based on our analysis, the convergence order should be $1$. The results presented in Figure \ref{fig:picard_convergence} verify this fact.

\vskip-15pt
\begin{figure}[!ht]
\centering
\subfigure[Convergence versus mesh size $h$ ($k = 0.01$ and $t=0.08$)]{ \includegraphics[width=0.45\textwidth]{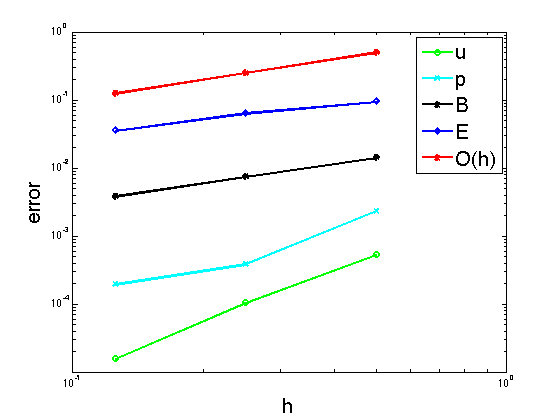} }
\qquad
\subfigure[Convergence versus time step size $k$ ($h = \nicefrac{1}{12}$ and $t=1$)]{ \includegraphics[width=0.45\textwidth]{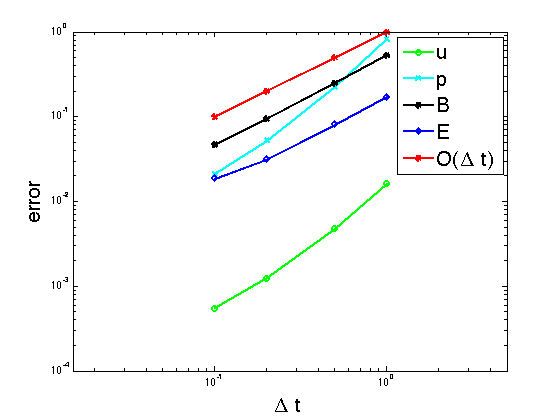} }
\qquad
\caption{Picard iteration: convergence test. }  \label{fig:picard_convergence}
\end{figure}

\section{Conclusions}
In this paper, we carry out error estimation of the structure-preserving discretization scheme proposed in \cite{Hu.K;Ma.Y;Xu.J.2014a}. These schemes achieves the optimal order of convergence. In addition, we confirm the theoretical analysis with numerical experiments.

\section*{Acknowledgements}

We would like to thank Shuonan Wu for his useful comments and suggestions.

\bibliographystyle{abbrv}
\bibliography{MHD}{}

\end{document}